\pdfoutput=1
\documentclass[11pt]{amsart}
\usepackage{latexsym,exscale,enumerate,enumitem,amsfonts,amssymb,mathtools}
\usepackage{amsmath,amsthm,amsfonts,amssymb,amscd,textcomp}
\usepackage{stmaryrd}
\SetSymbolFont{stmry}{bold}{U}{stmry}{m}{n}
\usepackage[normalem]{ulem}
\usepackage{thmtools}
\usepackage{easybmat}
\usepackage{young}
\usepackage{pdfpages}

\usepackage{tikz}
\usetikzlibrary{decorations.markings}
\usetikzlibrary{decorations.pathreplacing}
\usetikzlibrary{arrows,shapes,positioning}
\tikzstyle directed=[postaction={decorate,decoration={markings,
    mark=at position #1 with {\arrow{>}}}}]
\tikzstyle rdirected=[postaction={decorate,decoration={markings,
    mark=at position #1 with {\arrow{<}}}}]

\usepackage[all]{xy}
\SelectTips{cm}{}

\definecolor{orchid}{RGB}{143,40,194}
\definecolor{lava}{RGB}{207,16,32}

\def\cal#1{\mathcal{#1}}%

\newcommand{\sll}[1]{\mathfrak{sl}_{#1}}
\newcommand{\gll}[1]{\mathfrak{gl}_{#1}}

\newcommand{\soo}[1]{\mathfrak{so}_{#1}}
\newcommand{\fg}{\mathfrak{g}}
\newcommand{\fb}{\mathfrak{b}}
\newcommand{\fp}{\mathfrak{p}}
\newcommand{\fh}{\mathfrak{h}^*}
\newcommand{\VW}{\xy(0,0)*{\bigvee_d};(-1,0)*{\bigvee_{\phantom{d}}};\endxy}

\newcommand{\Uu}{\boldsymbol{\mathrm{U}}}
\newcommand{\Uv}{\boldsymbol{\mathrm{U}}_v}
\DeclareRobustCommand{\Uq}{\boldsymbol{\mathrm{U}}_q}

\newcommand{\Endr}{\mathrm{End}_{\Uq(\fg)}(T)}
\DeclareRobustCommand{\Endrr}{\mathrm{End}_{\Uq}(T)}
\DeclareRobustCommand{\Mod}[1]{{#1}\text{-}\boldsymbol{\mathrm{Mod}}}

\newcommand{\Hom}{\mathrm{Hom}}
\newcommand{\End}{\mathrm{End}}

\newcommand{\Exti}{\mathrm{Ext}^i_{\Uq}}
\newcommand{\Extii}{\mathrm{Ext}^1_{\Uq}}
\newcommand{\Char}{\mathrm{char}}

\newcommand{\Dl}{\Delta_q(\lambda)}
\newcommand{\Nl}{\nabla_q(\lambda)}
\newcommand{\Ll}{L_q(\lambda)}
\newcommand{\Tl}{T_q(\lambda)}
\newcommand{\Dm}{\Delta_q(\mu)}
\newcommand{\Nm}{\nabla_q(\mu)}

\newcommand{\TL}{\mathrm{GL}}

\newcommand{\T}{\boldsymbol{\mathcal{T}}}

\def\C{{\mathbb C}}
\def\N{{\mathbb{Z}_{\geq 0}}}
\def\Z{{\mathbb Z}}

\def\Q{{\mathbb Q}}
\def\K{{\mathbb K}}

\def\P{\mathcal{P}}
\def\Tt{\mathcal{I}}
\def\Ttt{\mathcal{I}^{\lambda}}
\def\Ct{\mathcal{C}}
\def\I{\mathrm{i}}

\theoremstyle{definition}
\newtheorem{thm}{Theorem}[section]
\newtheorem{cor}[thm]{Corollary}
\newtheorem{lem}[thm]{Lemma}
\newtheorem{prop}[thm]{Proposition}

\newtheorem{remark}{Remark}

\declaretheorem[style=definition,name=Assumption,qed=$\blacktriangle$,numberlike=thm]{as}

\declaretheorem[style=definition,name=Remark,qed=$\blacktriangle$,numberlike=remark]{rem}
\declaretheorem[style=definition,name=Example,qed=$\blacktriangle$,numberlike=thm]{ex}
\declaretheorem[style=definition,name=Definition,qed=$\blacktriangle$,numberlike=thm]{defn}

\newtheorem*{thmm}{Theorem}

\newcommand{\makeqed}{\hfill\ensuremath{\square}}
\newcommand{\qedmake}{\hfill\ensuremath{\blacksquare}}

\usepackage[final]{hyperref}
\usepackage{bookmark,cleveref}
\hypersetup{
    pdftoolbar=true,        
    pdfmenubar=true,        
    pdffitwindow=false,     
    pdfstartview={FitH},    
    pdftitle={Cellular structures using \texorpdfstring{$\Uq$}{Uq}-tilting modules},    
    pdfauthor={Henning Haahr Andersen, Catharina Stroppel and Daniel Tubbenhauer},     
    pdfsubject={},   
    pdfcreator={Henning Haahr Andersen, Catharina Stroppel and Daniel Tubbenhauer},   
    pdfproducer={Henning Haahr Andersen, Catharina Stroppel and Daniel Tubbenhauer}, 
    pdfkeywords={}, 
    pdfnewwindow=true,      
    colorlinks=true,       
    linkcolor=lava,          
    citecolor=teal,        
    filecolor=magenta,      
    urlcolor=orchid,          
    linkbordercolor=lava,
    citebordercolor=teal,
    urlbordercolor=orchid,  
    linktocpage=true
}
\let\fullref\autoref
\def\makeautorefname#1#2{\expandafter\def\csname#1autorefname\endcsname{#2}}
\makeautorefname{equation}{Equation}%
\makeautorefname{footnote}{footnote}%
\makeautorefname{item}{item}%
\makeautorefname{figure}{Figure}%
\makeautorefname{table}{Table}%
\makeautorefname{part}{Part}%
\makeautorefname{appendix}{Appendix}%
\makeautorefname{chapter}{Chapter}%
\makeautorefname{section}{Section}%
\makeautorefname{subsection}{Section}%
\makeautorefname{subsubsection}{Section}%
\makeautorefname{paragraph}{Paragraph}%
\makeautorefname{subparagraph}{Paragraph}%
\makeautorefname{theorem}{Theorem}%
\makeautorefname{theo}{Theorem}%
\makeautorefname{thm}{Theorem}%
\makeautorefname{addendum}{Addendum}%
\makeautorefname{addend}{Addendum}%
\makeautorefname{add}{Addendum}%
\makeautorefname{maintheorem}{Main theorem}%
\makeautorefname{mainthm}{Main theorem}%
\makeautorefname{corollary}{Corollary}%
\makeautorefname{corol}{Corollary}%
\makeautorefname{coro}{Corollary}%
\makeautorefname{cor}{Corollary}%
\makeautorefname{lemma}{Lemma}%
\makeautorefname{lemm}{Lemma}%
\makeautorefname{lem}{Lemma}%
\makeautorefname{sublemma}{Sublemma}%
\makeautorefname{sublem}{Sublemma}%
\makeautorefname{subl}{Sublemma}%
\makeautorefname{proposition}{Proposition}%
\makeautorefname{proposit}{Proposition}%
\makeautorefname{propos}{Proposition}%
\makeautorefname{propo}{Proposition}%
\makeautorefname{prop}{Proposition}%
\makeautorefname{property}{Property}
\makeautorefname{proper}{Property}
\makeautorefname{scholium}{Scholium}%
\makeautorefname{step}{Step}%
\makeautorefname{conjecture}{Conjecture}%
\makeautorefname{conject}{Conjecture}%
\makeautorefname{conj}{Conjecture}%
\makeautorefname{question}{Question}
\makeautorefname{questn}{Question}
\makeautorefname{quest}{Question}
\makeautorefname{ques}{Question}
\makeautorefname{qn}{Question}
\makeautorefname{definition}{Definition}%
\makeautorefname{defin}{Definition}%
\makeautorefname{defi}{Definition}%
\makeautorefname{def}{Definition}%
\makeautorefname{dfn}{Definition}%
\makeautorefname{notation}{Notation}
\makeautorefname{nota}{Notation}
\makeautorefname{notn}{Notation}
\makeautorefname{remark}{Remark}%
\makeautorefname{rema}{Remark}%
\makeautorefname{rem}{Remark}%
\makeautorefname{rmk}{Remark}%
\makeautorefname{rk}{Remark}%
\makeautorefname{remarks}{Remarks}%
\makeautorefname{rems}{Remarks}%
\makeautorefname{rmks}{Remarks}%
\makeautorefname{rks}{Remarks}%
\makeautorefname{example}{Example}%
\makeautorefname{examp}{Example}%
\makeautorefname{exmp}{Example}%
\makeautorefname{exam}{Example}%
\makeautorefname{exa}{Example}%
\makeautorefname{algorithm}{Algorith}%
\makeautorefname{algo}{Algorith}%
\makeautorefname{alg}{Algorith}%
\makeautorefname{axiom}{Axiom}%
\makeautorefname{axi}{Axiom}%
\makeautorefname{ax}{Axiom}%
\makeautorefname{case}{Case}%
\makeautorefname{claim}{Claim}%
\makeautorefname{clm}{Claim}%
\makeautorefname{assumption}{Assumption}%
\makeautorefname{assumpt}{Assumption}%
\makeautorefname{conclusion}{Conclusion}%
\makeautorefname{concl}{Conclusion}%
\makeautorefname{conc}{Conclusion}%
\makeautorefname{condition}{Condition}%
\makeautorefname{condit}{Condition}%
\makeautorefname{cond}{Condition}%
\makeautorefname{construction}{Construction}%
\makeautorefname{construct}{Construction}%
\makeautorefname{const}{Construction}%
\makeautorefname{cons}{Construction}%
\makeautorefname{criterion}{Criterion}%
\makeautorefname{criter}{Criterion}%
\makeautorefname{crit}{Criterion}%
\makeautorefname{exercise}{Exercise}%
\makeautorefname{exer}{Exercise}%
\makeautorefname{exe}{Exercise}%
\makeautorefname{problem}{Problem}%
\makeautorefname{problm}{Problem}%
\makeautorefname{probm}{Problem}%
\makeautorefname{prob}{Problem}%
\makeautorefname{solution}{Solution}%
\makeautorefname{soln}{Solution}%
\makeautorefname{sol}{Solution}%
\makeautorefname{summary}{Summary}%
\makeautorefname{summ}{Summary}%
\makeautorefname{sum}{Summary}%
\makeautorefname{operation}{Operation}%
\makeautorefname{oper}{Operation}%
\makeautorefname{observation}{Observation}%
\makeautorefname{observn}{Observation}%
\makeautorefname{obser}{Observation}%
\makeautorefname{obs}{Observation}%
\makeautorefname{ob}{Observation}%
\makeautorefname{convention}{Convention}%
\makeautorefname{convent}{Convention}%
\makeautorefname{conv}{Convention}%
\makeautorefname{cvn}{Convention}%
\makeautorefname{warning}{Warning}%
\makeautorefname{warn}{Warning}%
\makeautorefname{note}{Note}%
\makeautorefname{fact}{Fact}%
\usepackage{luatex85}
\usepackage{kantlipsum} 
\begin{document}
\vbadness=10001
\hbadness=10001
\title[Cellular structures using $\Uq$-tilting modules]{Cellular structures using $\Uq$-tilting modules\\{\tiny With additional notes to the paper as an Appendix}}

\author{Henning Haahr Andersen}

\author{Catharina Stroppel}

\author{Daniel Tubbenhauer}

\thanks{H.H.A. and D.T. were partially supported by the center of excellence grant ``Centre for Quantum Geometry of Moduli Spaces (QGM)'' from the ``Danish National Research Foundation (DNRF)''. C.S. was supported by the Max--Planck-Gesellschaft. D.T. was partially supported by a research funding of the ``Deutsche Forschungsgemeinschaft (DFG)'' during the last part of this work}

\begin{abstract}
We use the theory of $\Uq$-tilting modules 
to construct cellular bases for
centralizer 
algebras. Our methods are quite general and work for any quantum group 
$\Uq$ attached 
to a Cartan matrix 
and include the non-semisimple cases for $q$ being a root of unity and ground
fields of positive characteristic.
Our approach also
generalizes to certain 
categories containing infinite-dimensional modules.
As applications, we give a new 
semisimplicty criterion for centralizer 
algebras, and recover the cellularity of several known algebras  
(with partially new cellular bases) which all fit into our general setup.
\end{abstract}

\maketitle

\vspace*{-1em}

\tableofcontents

\vspace*{-1em}
%
%
\section{Introduction}\label{sec-intro}
Fix any field $\K$ and set $\K^{\ast}=\K-\{0,-1\}$ if $\Char(\K)>2$ and 
$\K^{\ast}=\K-\{0\}$ otherwise.
Let $\Uq(\fg)$ be the quantum group over $\K$
for a fixed, arbitrary 
parameter $q\in\K^{\ast}$ associated 
to a simple Lie algebra $\fg$. 
The main result in this 
paper is the following.

\begin{thmm}(\textbf{Cellularity of endomorphism algebras.})
Let $T$ be a $\Uq(\fg)$-tilting module. Then 
$\Endr$ is a cellular algebra in the sense of Graham and Lehrer \cite{gl}.\makeqed
\end{thmm}

It is important to note 
that cellular bases are not unique. In 
particular, a single algebra can have many 
cellular bases. As a concrete application, see 
\fullref{sub-graded}, we construct 
(several) new cellular bases for 
the Temperley--Lieb algebra depending on the ground field 
and the choice of deformation parameter. These bases differ therefore for 
instance from the construction in \cite[Section 6]{gl} of cellular 
bases for the Temperley--Lieb algebras. 
Moreover, we also show that some of our bases for the Temperley--Lieb 
algebra can be equipped with a $\Z$-grading which is in contrast 
to Graham and Lehrer's bases. Our bases also depend heavily 
on the characteristic of $\K$ (and on $q\in\K^{\ast}$). Hence, they see more of the 
characteristic (and parameter) depended representation theory,
but are also more difficult to construct explicitly.

We stress that the cellularity
itself can be deduced from general theory. Namely, any $\Uq(\fg)$-tilting module $T$ 
is a summand of a full $\Uq(\fg)$-tilting module $\tilde T$. 
By \cite[Theorem 6]{ring} 
$\End_{\Uq(\fg)}(\tilde T)$ is quasi-hereditary 
and comes equipped with an involution 
as we explain in \fullref{sub-celltilting}.
Thus, it is cellular, see \cite{kx}. 
By their Theorem 4.3, 
this induces the cellularity of the idempotent truncation $\Endr$.
In contrast, our approach provides the existence and a method 
of construction of many cellular bases. It generalizes to the 
infinite-dimensional Lie theory situation and has other nice consequences 
that we will explore in this paper.
In particular, our results give a novel semisimplicity 
criterion for $\Endr$, see \fullref{thm-cellsemisimple}. This 
together with the Jantzen 
sum formula give rise to a new way to obtain semisimplicity 
criteria for these algebras (we explain and explore this 
in \cite{ast} where we recover 
semisimplicity criteria for several algebras using the results of this paper).
Here a crucial fact is that the tensor 
product of $\Uq$-tilting modules is again a $\Uq$-tilting module, see \cite{par}. 
This implies that our results 
also vastly generalize \cite{west1} to 
the non-semisimple cases (where our main theorem is non-trivial).

\subsubsection*{The framework}\label{sub-intropart1}

Given any simple, complex Lie algebra 
$\fg$, 
we can assign to it a \textit{quantum deformation} $\Uv=\Uv(\fg)$ of its 
universal enveloping algebra 
by deforming 
its Serre presentation. (Here $v$ is a generic parameter and $\Uv$ is an $\Q(v)$-algebra.)
The representation theory of $\Uv$ shares many 
similarities with the one of $\fg$. In particular, the category\footnote{For 
any algebra $A$ we denote by $\Mod{A}$ the category of 
finite-dimensional, left $A$-modules. If not stated otherwise, all modules 
are assumed to be finite-dimensional, left modules.} $\Mod{\Uv}$ is semisimple.

But one can spice up the story drastically: the quantum group 
$\Uq=\Uq(\fg)$ is obtained by specializing 
$v$ to an arbitrary $q\in\K^{\ast}$. 
In particular, we 
can take $q$ to be a root of unity\footnote{In our terminology: The two cases 
$q=\pm 1$ are special and do not count as roots of unity. 
Moreover, for technical reasons, we always exclude $q=-1$ in case $\Char(\K)>2$.}. 
In this case $\Mod{\Uq}$ is not semisimple anymore, which makes 
the representation theory much more interesting. It has 
many connections and applications in different directions, e.g. 
the category has a neat combinatorics, is related to the corresponding 
almost-simple, simply connected algebraic group $G$ over $\K$ 
with $\Char(\K)$ prime, see for example \cite{ajs} or \cite{lu1}, 
to the representation theory of affine Kac--Moody algebras, 
see \cite{kalu} or \cite{tani},
and to $(2{+}1)$-TQFT's and the Witten--Reshetikhin--Turaev 
invariants of $3$-manifolds, see for example \cite{tur}.

Semisimplicity in light of our main result means the following. If we 
take $\K=\C$ and $q=\pm 1$, then our result says that
the algebra 
$\Endrr$ is cellular for any 
$\Uq$-module $T\in\Mod{\Uq}$ because in this case all $\Uq$-modules are 
$\Uq$-tilting modules. 
This is no surprise: when $T$ is a direct sum of simple $\Uq$-modules, then 
$\Endrr$ is a direct sum of matrix algebras $M_n(\K)$.
Likewise, for any $\K$, if $q\in\K^{\ast}-\{1\}$ is not a root of unity, then $\Mod{\Uq}$ 
is still semisimple and our result is (almost) standard. 
But even in the semisimple case we can say more: we get 
an Artin--Wedderburn basis
as a cellular basis for $\Endrr$, i.e. a 
basis realizing the decomposition of $\Endrr$ into its matrix components, 
see \fullref{sub-examples}.

On the other hand, if $q=1$ and $\Char(\K)>0$ or if $q\in\K^{\ast}$ is a root of unity, 
then $\Mod{\Uq}$ is far from being semisimple and our result 
gives many interesting cellular algebras.

For example, if $G=\mathrm{GL}(V)$ for some $n$-dimensional 
$\K$-vector space $V$, then $T=V^{\otimes d}$ is a $G$-tilting module 
for any $d\in\Z_{\geq 0}$. By 
Schur--Weyl duality we have
\begin{equation}\label{eq-schurweyl}
\Phi_{\mathrm{SW}}\colon\K[S_d]\twoheadrightarrow\End_G(T)\quad\text{and}
\quad\Phi_{\mathrm{SW}}\colon\K[S_d]\xrightarrow{\cong}\End_G(T),\text{ if }n\geq d,
\end{equation}
where $\K[S_d]$ is the group algebra of the symmetric group $S_d$ in $d$ letters. We can 
realize this as a special case in our framework by taking $q=1$, $n\geq d$ 
and $\fg=\mathfrak{gl}_n$ (although $\mathfrak{gl}_n$ is 
not a simple, complex Lie algebra, our approach works fine for it as well). On 
the other hand, by taking $q$ 
arbitrary in $\K^{\ast}-\{1\}$ and $n\geq d$, the 
group algebra $\K[S_d]$ is replaced by the type $A_{d-1}$ \textit{Iwahori--Hecke algebra} 
$\cal{H}_d(q)$ over $\K$ and 
our theorem gives cellular bases for this algebra 
as well. 
Note that one underlying fact 
why \eqref{eq-schurweyl} stays true in the non-semisimple 
case is that 
$\dim(\End_G(T))$ is independent of the 
characteristic of $\K$ 
(and of the parameter $q$ in the quantum case), since $T$ is a $G$-tilting module.

Of course, both $\K[S_d]$ and $\cal{H}_d(q)$ are known to be cellular (these cases were 
one of the main motivations of Graham and Lehrer to introduce the notion of cellular algebras), 
but the point we want to make is, 
that they fit into our more general framework.

The following known cellularity properties 
can also be recovered directly from our approach. 
And moreover: in most of the examples we either have 
no or only some mild restrictions on $\K$ and $q\in\K^{\ast}$.

\vspace{-0.5ex}
\begin{itemize}
\item As sketched above: the algebras 
$\K[S_d]$ and $\cal{H}_d(q)$ and their quotients under $\Phi_{\mathrm{SW}}$.
\item The \textit{Temperley--Lieb algebras} $\cal{TL}_d(\delta)$ 
introduced in \cite{tl}.
\item Other less well-known endomorphism algebras for $\sll{2}$-related 
tilting modules appearing in more 
recent work, e.g. \cite{alz}, \cite{at} or \cite{rt}.
\item \textit{Spider algebras} in the sense of \cite{kup}.
\item Quotients of the group algebras 
of $\Z/r\Z\wr S_d$ and its quantum version $\cal{H}_{d,r}(q)$, 
the \textit{Ariki--Koike algebras} introduced in \cite{ak}. 
This includes the Ariki--Koike algebras 
themselves and thus, the Hecke algebras of type $B$. This 
also includes 
Martin and Saleur's \textit{blob algebras} $\cal{BL}_d(q,m)$ \cite{ms} 
and \textit{(quantized) rook monoid algebras} (also called \textit{Solomon algebras}) 
$\cal{R}_d(q)$ in the spirit of \cite{sol}. 
\item \textit{Brauer algebras} 
$\cal{B}_d(\delta)$ introduced in \cite{bra} 
in the context of classical 
invariant theory, and related algebras, 
e.g. the \textit{walled 
Brauer algebras} $\cal{B}_{r,s}(\delta)$ as in \cite{koi} and \cite{tur1}, 
and the \textit{Birman--Murakami--Wenzl} algebras 
$\cal{BMW}_d(\delta)$, in the sense of \cite{bw} and \cite{mu}.
\end{itemize}
\vspace{-0.5ex}

Note our methods also apply for some categories 
containing infinite-dimensional modules. 
For example, with a little bit more care, one could allow $T$ to be 
a not necessarily finite-dimensional $\Uq$-tilting module.
Moreover, our methods also 
include the \textit{BGG category} $\boldsymbol{\cal{O}}$, 
its \textit{parabolic subcategories} $\boldsymbol{\cal{O}}^{\fp}$
and its quantum cousin $\boldsymbol{\cal{O}}_q$ from \cite{am}. For example, 
using the ``big projective tilting'' 
in the principal block, we get a cellular basis for 
the coinvariant algebra of the Weyl group associated to $\fg$. In fact, we get 
a vast generalization of this, e.g. we can fit 
\textit{generalized Khovanov arc algebras} 
(see e.g. \cite{bs1}), $\sll{n}$\textit{-web algebras} (see e.g. \cite{mpt}), 
\textit{cyclotomic Khovanov--Lauda and Rouquier algebras} of type $A$
(see \cite{kl1} and \cite{kl3} or \cite{rou}), 
for which we obtain 
cellularity via the connection 
to cyclotomic quotients of the degenerate 
affine Hecke algebra, see \cite{bk1},
cyclotomic $\VW$\textit{-algebras} 
(see e.g. \cite{es}) and \textit{cyclotomic quotients of 
affine Hecke algebras} $\boldsymbol{\mathrm{H}}^s_{\K,d}$ 
(see e.g. \cite{rsvv}) into our framework as well, see 
\fullref{sub-examples}.
However, we will for simplicity focus on the finite-dimensional world.
Here we provide all necessary arguments in great detail, 
sometimes, for brevity, only in an extra file \cite{astproofs}. 
See also \fullref{remark:qher}.

Following Graham and Lehrer's approach, our cellular bases for $\Endrr$ provide also 
$\Endrr$-cell modules, the classification of simple $\Endrr$-modules etc. We give 
an interpretation of this in our setting as well, see 
\fullref{sec-endoringsrep}. For instance, we 
deduce a new criterion for semisimplicity of $\Endrr$, see \fullref{thm-cellsemisimple}.

\begin{rem}\label{remark:qher}
Instead of working with the infinite-dimensional algebra $\Uq$, 
we could also work with a finite-dimensional, quasi-hereditary 
algebra (with a suitable anti-involution). By using 
results summarized in \cite[Appendix]{don1}, our constructions will 
go through very much in the same spirit as for $\Uq$. However, 
using $\Uq$ has some advantages.
For example, we can construct an abundance 
of cellular bases 
(for the explicit construction of
our basis we need ``weight spaces'' such 
that e.g. \eqref{eq-uniquemap} or \fullref{lem-cell1} work). 
Having several cellular bases 
is certainly an advantage, although 
calculating these is in general a non-trivial 
task. (For example, getting an explicit understanding of the 
endomorphisms giving rise to the cellular basis 
is a tough challenge, but see \cite{rw} for 
some crucial steps in this direction.)
As a direct consequence of the existence of many 
cellular bases: most of the algebras 
appearing in our list of examples above 
can be additionally equipped with a $\Z$-grading. The basis 
elements from \fullref{thm-cell} can be chosen 
such that our approach leads to a 
$\Z$-graded cellular basis in the sense of \cite{hm}. 
We make this more precise in case of 
the Temperley--Lieb algebras, 
but one could for instance also recover 
the $\Z$-graded cellular bases of the 
Brauer algebras from \cite{ehst} 
from our approach. 
We stress that in both cases the 
cellular bases in \cite[Sections 4 and 6]{gl} 
are not $\Z$-graded.
To keep the paper within reasonable 
boundaries, we do not treat the graded setup 
in detail.
\end{rem}
\vspace*{.5cm}
\paragraph*{\textbf{Acknowledgements}} We would like to thank Ben Cooper, 
Michael Ehrig, Matt Hogancamp, Johannes K\"ubel, 
Gus Lehrer, Paul Martin, Andrew Mathas, Volodymyr Mazorchuk, 
Steen Ryom-Hansen 
and Paul Wedrich
for helpful comments and discussions, and the referees for further useful comments. 
H.H.A. would like to thank the Institut Mittag-Leffler for 
the hospitality
he enjoyed there during the final stages of this work. C.S. 
is very grateful to the Max-Planck Institute in Bonn for the 
extraordinary support and excellent working 
conditions. A large part of her research was worked out during 
her stay there. D.T. would like to thank the dark 
Danish winter for very successfully limiting his non-work options 
and the Australian long blacks for pushing him forward.
%
\section{Quantum groups, their representations and tilting modules}\label{sec-qreps}
We briefly recall some facts 
we need in this paper. 
Details can be found e.g. in \cite{apw} and \cite{ja}, or \cite{don1} and \cite{jarag}. 
For notations and arguments 
adopted to our situation see \cite{astproofs}. 
See also \cite{ring} and \cite{don} for 
the classical treatment of tilting modules (in the modular 
case). As in the introduction, we fix a field $\K$ 
over which we work throughout.

\subsection{The quantum group \texorpdfstring{$\Uq$}{Uq}}\label{sub-quantumgroups}

Let $\Phi$ be a finite \textit{root system} in an Euclidean 
space $E$. We fix a choice of \textit{positive roots} $\Phi^+\subset\Phi$ 
and \textit{simple roots} $\Pi\subset\Phi^+$. We assume that we have $n$ 
simple roots that we denote by $\alpha_1,\dots,\alpha_n$. 
For each $\alpha\in\Phi$, we denote by 
$\alpha^{\vee}\in\Phi^{\vee}$ the corresponding \textit{coroot}.
Then $\mathbf{A}=(\langle\alpha_i,\alpha_j^{\vee}\rangle)_{i,j=1}^n$ is 
called the \textit{Cartan matrix}.

By the set of \textit{(integral) weights} 
we mean 
$X=\{\lambda \in E\mid \langle\lambda,\alpha_i^{\vee}\rangle\in\Z\text{ for all }\alpha_i\in\Pi\}$. The 
\textit{dominant (integral) weights} $X^+$ 
are those $\lambda\in X$ such that 
$\langle\lambda,\alpha_i^{\vee}\rangle\geq 0$ for all $\alpha_i\in\Pi$.

Recall that there is a \textit{partial ordering} on $X$ given by $\mu\leq \lambda$ 
if and only if $\lambda-\mu$ is an $\Z_{\geq 0}$-valued linear combination of the simple roots, that is, 
$\lambda-\mu=\sum_{i=1}^na_i\alpha_i$ with $a_i\in\Z_{\geq 0}$.

We denote by $\Uq=\Uq(\mathbf{A})$ 
the \textit{quantum enveloping algebra} 
attached to a Cartan matrix $\mathbf{A}$ and specialized at $q\in\K^{\ast}$, 
where we follow \cite{apw} with our conventions. 
Note $\Uq$ always means 
the quantum group over $\K$ defined via Lusztig's divided power 
construction. (Thus, 
we have generators $K_i,E_i$ and $F_i$ for all 
$i=1,\dots,n$ as well as divided power generators.) We have a
decomposition $\Uq=\Uq^-\Uq^0\Uq^+$, with subalgebras 
generated by $F$'s, $K$'s and $E$'s respectively 
(and some divided power generators, see e.g. their Section 1).
Note we can recover the generic 
case $\Uv=\Uv(\mathbf{A})$ by choosing $\K=\Q(v)$ and $q=v$.

It is worth noting that $\Uq$ is a Hopf algebra, so its module 
category is a monoidal category with duals. 
We denote by $\Mod{\Uq}$ the category 
of finite-dimensional $\Uq$-modules (of type $1$, see \cite[Section 1.4]{apw}). 
We consider only such $\Uq$-modules in what follows.

Recall that there 
is a contravariant, character-preserving \textit{duality functor}
$\cal{D}$ that is defined on the $\K$-vector space level 
via $\cal{D}(M)=M^*$ (the $\K$-linear dual 
of $M$) and an action of $\Uq$ on $\cal{D}(M)$ is defined 
as follows. Let $\omega\colon\Uq\to\Uq$ be the automorphism 
of $\Uq$ which interchanges $E_i$ and $F_i$ and interchanges $K_i$ and 
$K_i^{-1}$ (see e.g. \cite[Lemma 4.6]{ja}, which extends 
to our setup without difficulties). Then define
$uf=m\mapsto f(\omega(S(u))m)$
for $u\in \Uq,f\in \cal{D}(M),m\in M$.
Given any $\Uq$-homomorphism $f$ between $\Uq$-modules, we also write 
$\mathrm{i}(f)=\cal{D}(f)$. This duality gives rise to the involution 
in our cellular datum from \fullref{sub-celltilting}.

\begin{as}\label{as-tech}
If $q$ is a root of unity, then, to 
avoid technicalities, we assume 
that $q$ is a primitive root of unity of odd order $l$. 
A treatment of the even case, 
that can be used to repeat everything in this paper in the case where $l$ is even, 
can be found in \cite{an2}.
Moreover, in case of type $G_2$ we additionally assume that $l$ is prime to $3$.
\end{as}

For each $\lambda\in X^+$ there is 
a \textit{Weyl} $\Uq$\textit{-module} $\Dl$ and a \textit{dual Weyl} 
$\Uq$\textit{-module} $\Nl$ satisfying $\cal{D}(\Dl)\cong\Nl$.
The $\Uq$-module $\Dl$ has a unique simple head $\Ll$ 
which is the unique simple socle of $\Nl$.
Thus, there is a (up to scalars) 
unique $\Uq$-homomorphism
\begin{equation}\label{eq-uniquemap}
c^{\lambda}\colon \Dl\to \Nl\quad(\text{mapping head to socle}).
\end{equation}
This relies on the 
fact that $\Dl$ and $\Nl$ both 
have one-dimensional $\lambda$-weight spaces. The 
same fact implies that $\End_{\Uq}(\Ll)\cong\K$ for all $\lambda\in X^+$, 
see \cite[Corollary 7.4]{apw}. This last property 
fails for quasi-hereditary 
algebras in general when $\K$ is not algebraically closed.

\begin{thm}(\textbf{Ext-vanishing.})\label{thm-extfunctor}
We have for all $\lambda,\mu\in X^+$ that
\[
\Exti(\Dl,\Nm)\cong\begin{cases}\K c^{\lambda},&\text{if }i=0\text{ and }\lambda=\mu,\\
0,&\text{else}.\end{cases}\hspace{3.3cm}
\raisebox{-0.3cm}{\makeqed}
\hspace{-3.3cm}
\]
\end{thm}

We have to 
enlarge the category $\Mod{\Uq}$ by non-necessarily finite-dimensional 
$\Uq$-modules to have enough injectives such that the 
$\Exti$-functors make sense by using $q$-analogs 
arguments as in \cite[Part I, Chapter 3]{jarag}. However, 
$\Mod{\Uq}$ 
has enough injectives in characteristic zero, 
see \cite[Proposition 5.8]{an1} for a treatment of the non-semisimple cases.

\begin{proof}
Similar to the modular analog treated 
in \cite[Proposition II.4.13]{jarag} 
(a proof in our notation can be found in \cite{astproofs}).
\end{proof}

\subsection{Tilting modules and Ext-vanishing}\label{sub-tilting}

We say that a $\Uq$-module $M$ has a $\Delta_q$\textit{-filtration} 
if there exists some $k\in\N$ and 
a finite descending sequence of $\Uq$-submodules
\[
M=M_{0}\supset M_1\supset\cdots\supset M_{k^{\prime}}\supset\dots\supset M_{k-1}\supset M_k=0,
\]
such that $M_{k^{\prime}}/M_{k^{\prime}+1}\cong \Delta_q(\lambda_{k^{\prime}})$ 
for all $k^{\prime}=0,\dots,k-1$ and some $\lambda_{k^{\prime}}\in X^+$.
A $\nabla_q$\textit{-filtration} is defined similarly, 
but using a finite ascending sequence of $\Uq$-submodules 
and $\nabla_q(\lambda)$'s instead of $\Delta_q(\lambda)$'s. 
We denote by $(M:\Dl)$ and $(N:\Nl)$ the corresponding multiplicities, 
which are well-defined by \fullref{cor-extfunctor}.
Note that a $\Uq$-module $M$ has a $\Delta_q$-filtration if and only if its dual $\cal{D}(M)$ has a 
$\nabla_q$-filtration.

A corollary of the Ext-vanishing theorem is the following, 
whose proof is left to the reader 
or can be found in \cite{astproofs}. (Note that 
the proof of \fullref{cor-extfunctor} 
therein gives, in principle, 
a method to find and construct bases of $\Hom_{\Uq}(M,\Nl)$ and 
$\Hom_{\Uq}(\Dl,N)$ respectively.)

\begin{cor}\label{cor-extfunctor}
Let $M,N\in\Mod{\Uq}$ and $\lambda\in X^+$. Assume that 
$M$ has a $\Delta_q$-filtration and $N$ has a $\nabla_q$-filtration. Then
\[
\dim(\Hom_{\Uq}(M,\Nl))=(M:\Dl)
\quad\text{and}\quad
\dim(\Hom_{\Uq}(\Dl,N))=(N:\Nl).
\]
In particular, $(M:\Dl)$ 
and $(N:\Nl)$ are independent 
of the choice of filtrations.\qedmake
\end{cor}

\begin{prop}(\textbf{Donkin's Ext-criteria.})\label{prop-extfunctor2}
The following are equivalent.

\begin{enumerate}[label=(\alph*)]
\item An $M\in\Mod{\Uq}$ has a $\Delta_q$-filtration
(respectively $N\in\Mod{\Uq}$ has a $\nabla_q$-filtration).
\item We have $\Exti(M,\Nl)=0$ 
(respectively $\Exti(\Dl,N)=0$)
for all $\lambda\in X^+$ and all $i>0$.
\item We have $\Extii(M,\Nl)=0$ 
(respectively $\Extii(\Dl,N)=0$)
for all $\lambda\in X^+$.\makeqed
\end{enumerate}
\end{prop}

\begin{proof}
As in \cite[Proposition II.4.16]{jarag}. 
A proof in our notation can be found in \cite{astproofs}.
\end{proof}

A $\Uq$-module $T$ which has both, a 
$\Delta_q$- and a $\nabla_q$-filtration, is called a 
$\Uq$\textit{-tilting module}.
Following Donkin \cite{don}, we are now ready to define the 
\textit{category of} $\Uq$\textit{-tilting modules} that we denote by $\T$. This 
category is our main object of study.

\begin{defn}(\textbf{Category of $\Uq$-tilting modules.})\label{defn-tiltcat}
The category $\T$ is the full subcategory 
of $\Mod{\Uq}$ whose objects are given by all $\Uq$-tilting modules.
\end{defn}

From \fullref{prop-extfunctor2} we 
obtain directly an important statement.

\begin{cor}\label{cor-tiltingext1}
Let $T\in\Mod{\Uq}$. Then
\[
T\in\T\quad
\text{if and only if}
\quad\Extii(T,\Nl)=0=\Extii(\Dl,T)\quad\text{for all }\lambda\in X^+.
\]
When $T\in\T$, the corresponding higher Ext-groups vanish as well.\qedmake
\end{cor}

The 
indecomposable $\Uq$-modules in $\T$, 
that we denote by $\Tl$, are 
indexed by $\lambda\in X^+$.
The
$\Uq$-tilting module $\Tl$ is determined by the 
property that it is indecomposable with 
$\lambda$ as its unique maximal 
weight. In fact, $(\Tl:\Dl)=1$, and 
$(\Tl:\Dm)\neq 0$ only if $\mu\leq\lambda$. (Dually for $\nabla_q$-filtrations.)

Note that the duality functor $\cal{D}$ from above restricts 
to $\T$. Moreover, as a consequence of the classification 
of indecomposable $\Uq$-modules in $\T$, 
we have $\cal{D}(T)\cong T$ for $T\in\T$. 
In particular, we have for all $\lambda\in X^+$ that
\[
(T:\Dl)=\dim(\Hom_{\Uq}(T,\Nl))=\dim(\Hom_{\Uq}(\Dl,T))=(T:\Nl).
\]
It is known that $\T$ is a Krull--Schmidt category, closed under 
finite direct sums, taking summands and finite tensor products 
(the latter is a non-trivial fact, see \cite[Theorem 3.3]{par}).

For a fixed $\lambda\in X^+$ we have $\Uq$-homomorphisms
\[
\begin{gathered}
\xymatrix{
  \Dl  \ar@{^{(}->}[r]^{\iota^{\lambda}}    &  \Tl  \ar@{->>}[r]^{\pi^{\lambda}}  &  \Nl,
}
\end{gathered}
\]
where $\iota^{\lambda}$ is the inclusion of the first 
$\Uq$-submodule in a $\Delta_q$-filtration of 
$\Tl$ and $\pi^{\lambda}$ is the surjection 
onto the last quotient in a $\nabla_q$-filtration of 
$\Tl$. Note that these are only defined up 
to scalars and we fix scalars in the following such that
$\pi^{\lambda}\circ\iota^{\lambda}=c^{\lambda}$ (where $c^{\lambda}$ is 
again the $\Uq$-homomorphism from \eqref{eq-uniquemap}).

\begin{rem}\label{rem-new}
Let $T\in\T$. 
An easy argument (based on \fullref{thm-extfunctor}) shows
the following crucial fact:
\begin{equation}\label{eq-important}
\Extii(\Dl,T)\hspace*{-0.3ex}=\hspace*{-0.3ex}0\hspace*{-0.3ex}=\hspace*{-0.3ex}\Extii(T,\Nl)\Rightarrow \Extii(\mathrm{coker}(\iota^{\lambda}),T)\hspace*{-0.3ex}=\hspace*{-0.3ex}0\hspace*{-0.3ex}=\hspace*{-0.3ex}\Extii(T,\ker(\pi^{\lambda}))
\end{equation}
for all $\lambda\in X^+$.
Consequently, we see that any $\Uq$-homomorphism $g\colon \Dl\to T$ 
extends to a $\Uq$-homomorphism $\overline{g}\colon \Tl\to T$ 
whereas any $\Uq$-homomorphism $f\colon T\to \Nl$ factors 
through $\Tl$ via some $\overline{f}\colon T\to\Tl$.
\end{rem}

\begin{rem}\label{remark:calc-character}
In \cite{astproofs} it is described in detail how to compute 
$(\Tl:\Dm)$ 
for $\lambda,\mu\in X^+$. This can be done algorithmically 
in case $q$ is a 
complex, primitive $l$-th root of unity, i.e. 
one can use Soergel's version of 
the \textit{affine parabolic Kazhdan--Lusztig polynomials}. 
For brevity, we do not recall the definition of these polynomials here, but refer 
to \cite[Section 3]{soe3} where the relevant polynomials 
are denoted $n_{y,x}$ (and where all the other relevant notions are defined). 
The main point for us is the following theorem due to Soergel \cite[Theorem 5.12]{soe4} 
(see also \cite[Conjecture 7.1]{soe3}): 
Suppose $\K=\C$ and $q$ is a complex, primitive $l$-th root of unity.
For each pair $\lambda,\mu\in X^+$ with $\lambda$ being an $l$-regular $\Uq$-weight 
(that is, $\Tl$ belongs to a regular block of $\T$) we have 
(with $n_{\mu\lambda}$ equal to the relevant $n_{y,x}$)
\[
(\Tl:\Dm)=n_{\mu\lambda}(1)=(\Tl:\Nm).
\]
From this one obtains a method to find 
the indecomposable summands of
$\Uq$-tilting modules
with known characters (e.g. tensor products of minuscule representations).
\end{rem}

%
%
\section{Cellular structures on endomorphism algebras}\label{sec-endorings}
In this section we give our 
construction of cellular bases for endomorphism 
rings $\Endrr$ of $\Uq$-tilting modules $T$ 
and prove our main result, that is, \fullref{thm-cell}.

The main tool is \fullref{thm-baseshom2}. The proof of the 
latter needs several ingredients which we establish in 
the form of separate lemmas 
collected in \fullref{sub-celllem}.

\subsection{The basis theorem}\label{sub-homspaces}

As before, we consider the category $\Mod{\Uq}$. Moreover, we fix 
two $\Uq$-modules $M,N$, where 
we assume that $M$ has a $\Delta_q$-filtration and $N$ has a 
$\nabla_q$-filtration. Then, by \fullref{cor-extfunctor}, we have
\begin{equation}\label{eq-dimhom}
\dim(\Hom_{\Uq}(M,N))=\sum_{\lambda\in X^+}(M:\Dl)(N:\Nl).
\end{equation}
We point out that the sum in \eqref{eq-dimhom} is actually finite 
since $(M:\Dl)\neq 0$ for only a finite number of $\lambda\in X^+$.
(Dually, $(N:\Nl)\neq 0$ for only finitely many $\lambda\in X^+$.)

Given $\lambda\in X^+$, we define for $(N:\Nl)>0$ respectively for $(M:\Dl)>0$ the two sets
\[
\cal{I}^{\lambda}=\{1,\dots,(N:\Nl)\}
\quad\text{and}\quad \cal{J}^{\lambda}=\{1,\dots,(M:\Dl)\}.
\]
By convention, $\cal{I}^{\lambda}=\emptyset$ and $\cal{J}^{\lambda}=\emptyset$ if 
$(N:\Nl)=0$ respectively if $(M:\Dl)=0$.

We can fix a basis 
of $\Hom_{\Uq}(M,\Nl)$
indexed by $\cal{J}^{\lambda}$. We denote this fixed basis by
$F^{\lambda}=\{f_j^{\lambda}\colon M\to\Nl\mid j\in\cal{J}^{\lambda}\}$.
By \fullref{prop-extfunctor2} and \eqref{eq-important},
we see that all elements of $F^{\lambda}$ factor through 
the $\Uq$-tilting module $\Tl$, i.e. we have commuting diagrams
\[
\begin{xy}
  \xymatrix{
      M \ar[rd]_/-0.25em/{f_j^{\lambda}}\ar@{.>}[r]^/-0.25em/{\exists\;\overline{f}_j^{\lambda}}    &   \Tl \ar@{->>}[d]^/-0.25em/{\pi^{\lambda}}  \\
        &   \Nl.   
  }
\end{xy}
\]
We call $\overline{f}_j^{\lambda}$ a
\textit{lift} of $f_j^{\lambda}$. 
(Note that a lift $\overline{f}_j^{\lambda}$ is not unique.)
Dually, we can choose a basis 
of $\Hom_{\Uq}(\Dl,N)$ as $G^{\lambda}=\{g_i^{\lambda}\colon \Dl\to N\mid i\in\cal{I}^{\lambda}\}$,
which extends to give 
(a non-unique) lift $\overline{g}_i^{\lambda}\colon \Tl\to N$ such that 
$\overline{g}_i^{\lambda}\circ \iota^{\lambda}=g_i^{\lambda}$ 
for all $i\in\cal{I}^{\lambda}$.

We can use this setup to define a basis for $\Hom_{\Uq}(M,N)$ 
which, when $M=N$, turns out to be a cellular basis, see \fullref{thm-cell}. 
For each $\lambda\in X^+$ and all $i\in\cal{I}^{\lambda},j\in\cal{J}^{\lambda}$ set
\[
c_{ij}^{\lambda}=\overline{g}_i^{\lambda}\circ \overline{f}_j^{\lambda}\in\Hom_{\Uq}(M,N).
\]
Our main result here is now the following.

\begin{thm}(\textbf{Basis theorem.})\label{thm-baseshom2}
For any choice of $F^{\lambda}$ and $G^{\lambda}$ 
as above and any choice of 
lifts of the $f_j^{\lambda}$'s and the $g_i^{\lambda}$'s (for all $\lambda\in X^+$), the set
\[
GF=\{c_{ij}^{\lambda}
\mid \lambda\in X^+,\; i\in\cal{I}^{\lambda},j\in\cal{J}^{\lambda}\}
\]
is a basis of $\Hom_{\Uq}(M,N)$.\makeqed
\end{thm}

\begin{proof}
This follows from \fullref{prop-baseshom} combined with \fullref{lem-cell4} 
and \fullref{lem-cell5}. 
\end{proof}

The basis $GF$ for $\Hom_{\Uq}(M,N)$ can be illustrated in a commuting diagram as
\[
\begin{gathered}
\begin{xy}
  \xymatrix{
  	      &   \Dl \ar@{^{(}->}[d]_{\iota^{\lambda}}\ar[rd]^{g_i^{\lambda}} &  \\
      M \ar@{.>}[r]^/-0.25em/{\overline{f}_j^{\lambda}}\ar[rd]_{f_j^{\lambda}}     &   \Tl \ar@{->>}[d]^{\pi^{\lambda}} \ar@{.>}[r]_/0.25em/{\overline{g}_i^{\lambda}}&  N\\
                   &  \Nl   &
  }
\end{xy}
\end{gathered}.
\]
Since $\Uq$-tilting modules have both a $\Delta_q$- and a $\nabla_q$-filtration, we get as 
an immediate consequence a
key result for our purposes.
\begin{cor}\label{cor-baseshom2}
Let $T\in\T$. Then $GF$ is, for any choices involved, a basis of $\Endrr$.\qedmake
\end{cor}

\subsection{Proof of the basis theorem}\label{sub-celllem}

We first show that, given lifts $\overline{f}_j^{\lambda}$, there is a consistent choice 
of lifts $\overline{g}_i^{\lambda}$ such that $GF$ is a basis 
of $\Hom_{\Uq}(M,N)$.

\begin{prop}(\textbf{Basis theorem --- dependent version.})\label{prop-baseshom}
For any choice of $F^{\lambda}$ and any choice of 
lifts of the $f_j^{\lambda}$'s (for all $\lambda\in X^+$) there exist a 
choice of a basis $G^{\lambda}$ 
and a choice of 
lifts of the $g_i^{\lambda}$'s such that
$
GF=\{c_{ij}^{\lambda}
\mid \lambda\in X^+,\; i\in\cal{I}^{\lambda},j\in\cal{J}^{\lambda}\}
$
is a basis of $\Hom_{\Uq}(M,N)$.\makeqed
\end{prop}

The corresponding statement with the roles of $f$'s and $g$'s 
swapped clearly holds as well.

\begin{proof}
We will construct $GF$ inductively.
For this purpose, let
\[
0=N_0\subset N_1\subset \cdots \subset N_{k-1} \subset N_k=N
\]
be a $\nabla_q$-filtration of $N$, i.e.
$N_{k^{\prime}+1}/N_{k^{\prime}}\cong 
\nabla_q(\lambda_{k^{\prime}})$ for some $\lambda_{k^{\prime}}\in X^+$ 
and all $k^{\prime}=0,\dots,k-1$.

Let $k=1$ and $\lambda_1=\lambda$. Then $N_1=\Nl$ and 
$\{c^{\lambda}\colon \Dl\to\Nl\}$ gives a basis of $\Hom_{\Uq}(\Dl,\Nl)$, 
where $c^{\lambda}$ is again the $\Uq$-homomorphism 
chosen in \eqref{eq-uniquemap}. Set 
$g^{\lambda}_1=c^{\lambda}$ and observe that 
$\overline{g}^{\lambda}_1=\pi^{\lambda}$ satisfies $\overline{g}^{\lambda}_1\circ \iota^{\lambda}=g^{\lambda}_1$. Thus, we have a basis
and a corresponding lift. This clearly gives a basis of 
$\Hom_{\Uq}(M,N_1)$, since, by assumption, we have that 
$F^{\lambda}$ gives a basis of $\Hom_{\Uq}(M,\Nl)$ and
$\pi^{\lambda}\circ \overline{f}_j^{\lambda}=f_j^{\lambda}$.

Hence, it remains to consider the case $k>1$. Set $\lambda_{k}=\lambda$ and 
observe that we have a short exact sequence of the form
\begin{equation}\label{eq-shortseq1}
\xymatrix{
  0\ar[r] & N_{k-1}  \ar@{^{(}->}[r]^/0.05cm/{\mathrm{inc}}    &  N_{k}  
  \ar@{->>}[r]^/-0.15cm/{\mathrm{pro}}  &  \Nl\ar[r] & 0.
}
\end{equation}
By \fullref{thm-extfunctor} (and the usual implication as 
in \eqref{eq-important}) 
this leads to a short exact sequence
\begin{equation}\label{eq-shortseq2}
\xymatrixcolsep{1.5pc}
\xymatrix{
  0\ar[r] & \Hom_{\Uq}(M,N_{k-1})  \ar@{^{(}->}[r]^/0.1cm/{\mathrm{inc}_*}    &  \Hom_{\Uq}(M,N_{k}) \ar@{->>}[r]^/-0.15cm/{\mathrm{pro}_*}  &  \Hom_{\Uq}(M,\Nl)\ar[r] & 0.
}
\end{equation}
By induction, we get from \eqref{eq-shortseq2} for all $\mu\in X^+$ a basis 
of $\Hom_{\Uq}(\Dm,N_{k-1})$ consisting of 
$g_i^{\mu}$'s with lifts $\overline{g}_i^{\mu}$ such that
\begin{gather}\label{eq:basis-included}
\{
c_{ij}^{\mu}=\overline{g}_i^{\mu}\circ \overline{f}_j^{\mu}\mid \mu\in X^+,\; i\in\cal{I}^{\mu}_{k-1},
j\in\cal{J}^{\mu}
\}
\end{gather}
is a basis of $\Hom_{\Uq}(M,N_{k-1})$ (here we use 
$\cal{I}^{\mu}_{k-1}=\{1,\dots,(N_{k-1}:\nabla_q(\mu))\}$). 
We define $g_i^{\mu}(N_k)=\mathrm{inc}\circ g_i^{\mu}$ and $\overline{g}_i^{\mu}(N_{k})=\mathrm{inc}\circ\overline{g}_i^{\mu}$ for each 
$\mu\in X^+$ and each $i\in\cal{I}^{\mu}_{k-1}$.

We now have to consider two cases, namely $\lambda\neq \mu$ and $\lambda=\mu$. In the 
first case we see that $\Hom_{\Uq}(\Delta_q(\mu),\Nl)=0$, so that, by 
using \eqref{eq-shortseq1} and the usual implication from \eqref{eq-important},
\[
\Hom_{\Uq}(\Delta_q(\mu),N_{k-1})\cong\Hom_{\Uq}(\Delta_q(\mu),N_{k}).
\]
Thus, our basis from \eqref{eq:basis-included} gives a basis of 
$\Hom_{\Uq}(\Delta_q(\mu),N_{k})$ and also gives the corresponding lifts.
On the other hand, if $\lambda=\mu$, then 
\[
(N_k:\Nl)=(N_{k-1}:\Nl)+1.
\] 
By \fullref{thm-extfunctor} (and the corresponding 
implication as in \eqref{eq-important}), 
we can 
choose $g^{\lambda}\colon\Dl\to N_{k}$ such that 
$\mathrm{pro}\circ g^{\lambda}=c^{\lambda}$. Then any choice of a 
lift $\overline{g}^{\lambda}$ of $g^{\lambda}$ will satisfy 
$\mathrm{pro}\circ \overline{g}^{\lambda}=\pi^{\lambda}$.

Adjoining $g^{\lambda}$ to the basis from \eqref{eq:basis-included} gives 
a basis of $\Hom_{\Uq}(\Delta_q(\lambda),N_{k})$ which satisfies the 
lifting property.
Note that we know from the case $k=1$ that
\[
\{
\mathrm{pro}\circ\overline{g}^{\lambda}\circ \overline{f}_j^{\lambda}=\pi^{\lambda}\circ \overline{f}_j^{\lambda}\mid j\in\cal{J}^{\lambda}
\}
\]
is a basis of $\Hom_{\Uq}(M,\Nl)$. 
Combining everything: we have that
\[
\{
c_{ij}^{\lambda}=\overline{g}_i^{\lambda}(N_{k})\circ \overline{f}_j^{\lambda}\mid \lambda\in X^+,\; i\in\cal{I}^{\lambda},
j\in\cal{J}^{\lambda}
\}
\]
is a basis of $\Hom_{\Uq}(M,N_{k})$ (by enumerating $\overline{g}_{(N:\Nl)}^{\lambda}(N_k)=\overline{g}^{\lambda}$ in the $\lambda=\mu$ case).
\end{proof}

We assume in the following that we have 
fixed some choices as in \fullref{prop-baseshom}.

Let $\lambda\in X^+$. Given $\varphi\in\Hom_{\Uq}(M,N)$, we denote by 
$\varphi_{\lambda}\in\Hom_{\Uq^0}(M_{\lambda},N_{\lambda})$ the 
induced $\Uq^0$-homomorphism (that is, $\K$-linear maps) between the 
$\lambda$-weight spaces $M_{\lambda}$ and $N_{\lambda}$. In addition, we denote by 
$\Hom_{\K}(M_{\lambda},N_{\lambda})$ the $\K$-linear maps between 
these $\lambda$-weight spaces.

\begin{lem}\label{lem-cell1}
For any $\lambda\in X^+$ the induced set
$
\{
(c_{ij}^{\lambda})_{\lambda}\mid c_{ij}^{\lambda}\in GF
\}
$
is a linearly independent subset of $\Hom_{\K}(M_{\lambda},N_{\lambda})$.\makeqed
\end{lem}

\begin{proof}
We proceed as in the proof of \fullref{prop-baseshom}.

If $N=\Nl$ (this was $k=1$ above), then  
$c_{1j}^{\lambda}=\pi^{\lambda}\circ \overline{f}_j^{\lambda}=f_j^{\lambda}$ and  
the $c_{1j}^{\lambda}$'s form a basis of $\Hom_{\Uq}(M,\Nl)$. By 
the $q$-Frobenius reciprocity from \cite[Proposition 1.17]{apw} we have
\[
\Hom_{\Uq}(M,\Nl)\cong\Hom_{\Uq^-\Uq^0}(M,\K_{\lambda})\subset \Hom_{\Uq^0}(M,\K_{\lambda})=\Hom_{\K}(M_{\lambda},\K).
\]
Hence, because $N_{\lambda}=\K$ in this case, we have the base of 
the induction.

Assume now $k>1$. The construction of $\{c_{ij}^{\mu}(N_k)\}_{\mu,i,j}$ in 
the proof of \fullref{prop-baseshom} shows that this set consists of 
two separate parts: one being the basis 
from \eqref{eq:basis-included} coming from a basis for 
$\Hom_{\Uq}(M,N_{k-1})$ and the second part (which only occurs when
$\lambda=\mu$) coming from a basis from Hom $\Hom_{\Uq}(\Delta_q(\lambda),N_k)$.

By \eqref{eq-shortseq2} there is a short exact sequence
\[
\xymatrix{
  0\ar[r] & \Hom_{\K}(M_{\lambda},(N_{k-1})_{\lambda})  \ar@{^{(}->}[r]^/0.1cm/{\mathrm{inc}_*}    &  \Hom_{\K}(M_{\lambda},(N_k)_{\lambda}) \ar@{->>}[r]^/0.10cm/{\mathrm{pro}_*}  &  \Hom_{\K}(M_{\lambda},\K)\ar[r] & 0.
}
\]
Thus, we can proceed as in the proof of \fullref{prop-baseshom}.
\end{proof}

We need another piece of notation: we define for each $\lambda\in X^+$
\[
\Hom_{\Uq}(M,N)^{\leq \lambda}=\{\varphi\in\Hom_{\Uq}(M,N)\mid \varphi_{\mu}=0\text{ unless }\mu\leq \lambda\}.
\]
In words: a $\Uq$-homomorphism $\varphi\in\Hom_{\Uq}(M,N)$ belongs 
to $\Hom_{\Uq}(M,N)^{\leq \lambda}$ if and only if $\varphi$ vanishes on all $\Uq$-weight 
spaces $M_{\mu}$ with $\mu\not\leq\lambda$. In addition to the notation above, we use 
the evident notation $\Hom_{\Uq}(M,N)^{< \lambda}$. We arrive at the following.

\begin{lem}\label{lem-cell3}
For any fixed $\lambda\in X^+$ the sets
\[
\{
c_{ij}^{\mu}\mid c_{ij}^{\mu}\in GF,\;\mu\leq \lambda
\}\quad\text{and}\quad \{
c_{ij}^{\mu}\mid c_{ij}^{\mu}\in GF,\;\mu< \lambda
\}
\]
are bases of $\Hom_{\Uq}(M,N)^{\leq \lambda}$ and 
$\Hom_{\Uq}(M,N)^{< \lambda}$ respectively.\makeqed
\end{lem}

\begin{proof}
As $c_{ij}^{\mu}$ factors through $T_q(\mu)$ 
and $T_q(\mu)_{\nu}=0$ unless $\nu\leq \mu$ 
(which follows using the classification of indecomposable
$\Uq$-tilting modules), we see 
that $(c_{ij}^{\mu})_{\nu}=0$ unless $\nu\leq \mu$. Moreover, by \fullref{lem-cell1}, 
each $(c_{ij}^{\mu})_{\mu}$ is non-zero. Thus, 
$c_{ij}^{\mu}\in \Hom_{\Uq}(M,N)^{\leq \lambda}$ if and only if $\mu\leq\lambda$.
Now choose any $\varphi\in\Hom_{\Uq}(M,N)^{\leq \lambda}$. By \fullref{prop-baseshom} we may write
\begin{equation}\label{eq-maybeuseful}
\varphi=\sum_{\mu,i,j}a_{ij}^{\mu}c_{ij}^{\mu},\quad a_{ij}^{\mu}\in\K.
\end{equation}
Choose $\mu\in X^+$ maximal with the property that there 
exist $i\in\cal{I}^{\lambda},j\in\cal{J}^{\lambda}$ such that $a_{ij}^{\mu}\neq 0$.

We claim that $a_{ij}^{\nu}(c_{ij}^{\nu})_{\mu}=0$ whenever $\nu\neq \mu$. 
This is true because, as observed above, $(c_{ij}^{\nu})_{\mu}=0$ unless 
$\mu\leq \nu$, and for $\mu<\nu$ we have $a_{ij}^{\nu}=0$ by the maximality of $\mu$.
We conclude $\varphi_{\mu}=\sum_{i,j}a_{ij}^{\mu}(c_{ij}^{\mu})_{\mu}$ and thus, 
$\varphi_{\mu}\neq 0$ by \fullref{lem-cell1}. Hence, $\mu\leq \lambda$, which gives 
by \eqref{eq-maybeuseful} that 
$\varphi\in\mathrm{span}_{\K}\{c_{ij}^{\mu}\mid c_{ij}^{\mu}\in GF,\;\mu\leq \lambda
\}$ 
as desired. This shows that 
$\{c_{ij}^{\mu}\mid c_{ij}^{\mu}\in GF,\;\mu\leq \lambda\}$ 
spans $\Hom_{\Uq}(M,N)^{\leq \lambda}$.
Since it is clearly a linearly independent set, it is a basis.

The second statement follows analogously, so the details are omitted.
\end{proof}

We need the following two lemmas to prove that all 
choices in \fullref{prop-baseshom} 
lead to bases of $\Hom_{\Uq}(M,N)$. As before 
we assume that we have, as in \fullref{prop-baseshom}, 
constructed $\{g^{\lambda}_i,i\in\cal{I}^{\lambda}\}$ and the corresponding 
lifts $\overline{g}^{\lambda}_i$ for all $\lambda\in X^+$.

\begin{lem}\label{lem-cell4}
Suppose that we have other $\Uq$-homomorphisms 
$\tilde g^{\lambda}_i\colon \Tl\to N$  
such that $\tilde g^{\lambda}_i\circ\iota^{\lambda}=g^{\lambda}_i$. Then 
the following set is also a basis of $\Hom_{\Uq}(M,N)$:
\[
\{
\tilde c_{ij}^{\lambda}=\tilde g^{\lambda}_i
\circ \overline{f}^{\lambda}_j\mid \lambda\in X^+,\; i\in\cal{I}^{\lambda}, j\in\cal{J}^{\lambda}
\}.
\hspace{4.225cm}
\makeqed
\hspace{-4.225cm}
\]
\end{lem}

\begin{proof}
As $(\overline{g}^{\lambda}_i-\tilde g^{\lambda}_i)\circ\iota^{\lambda}=0$, we see 
that $\overline{g}^{\lambda}_i-\tilde g^{\lambda}_i\in \Hom_{\Uq}(\Tl,N)^{< \lambda}$. 
Hence, we 
have $c_{ij}^{\lambda}-\tilde c_{ij}^{\lambda}\in \Hom_{\Uq}(M,N)^{< \lambda}$. Thus, 
by \fullref{lem-cell3}, there is a unitriangular 
change-of-basis matrix between $\{c_{ij}^{\lambda}\}_{\lambda,i,j}$ and $\{\tilde c_{ij}^{\lambda}\}_{\lambda,i,j}$.
\end{proof}

Now assume that we have chosen another basis 
$\{h^{\lambda}_i\mid i\in\cal{I}^{\lambda}\}$ of the spaces $\Hom_{\Uq}(\Dl,N)$ 
for each $\lambda\in X^+$ and the corresponding 
lifts $\overline{h}^{\lambda}_i$ as well.

\begin{lem}\label{lem-cell5}
The following set is also a basis of $\Hom_{\Uq}(M,N)$:
\[
\{
d_{ij}^{\lambda}=\overline{h}^{\lambda}_i
\circ \overline{f}^{\lambda}_j\mid \lambda\in X^+,\; i\in\cal{I}^{\lambda}, j\in\cal{J}^{\lambda}
\}.
\hspace{4.2cm}
\makeqed
\hspace{-4.2cm}
\]
\end{lem}

\begin{proof}
Write $g_i^{\lambda}=\sum_{k=1}^{(N:\Nl)}b^{\lambda}_{ik}h_k^{\lambda}$ 
with $b^{\lambda}_{ik}\in\K$ and 
set $\tilde{g}_i^{\lambda}=\sum_{k=1}^{(N:\Nl)}b^{\lambda}_{ik}\overline{h}_k^{\lambda}$. 
Then 
the $\tilde{g}_i^{\lambda}$'s are lifts of the $g_i^{\lambda}$'s. 
Hence, by \fullref{lem-cell4}, 
the elements $\tilde{g}_i^{\lambda}\circ\overline{f}^{\lambda}_j$ 
form a basis of $\Hom_{\Uq}(M,N)$. Thus, this proves the lemma, 
since, 
by construction, $\{d_{ij}^{\lambda}\}_{\lambda,i,j}$ is related 
to this basis by the 
invertible change-of-basis matrix $(b_{ik}^{\lambda})_{i,k=1;\lambda\in X^+}^{(N:\Nl)}$.
\end{proof}

In total, we established \fullref{prop-baseshom}.

\subsection{Cellular structures on endomorphism algebras of \texorpdfstring{$\Uq$}{Uq}-tilting modules}\label{sub-celltilting}

This subsection finally contains the statement and proof of our main theorem. 
We keep on working 
over a field $\K$ instead of a ring as for example Graham and Lehrer \cite{gl} do. 
(This avoids technicalities, e.g. the theory of indecomposable $\Uq$-tilting 
modules over rings is much more subtle 
than over fields. See e.g. \cite[Remark 1.7]{don}.)

\begin{defn}(\textbf{Cellular algebras.})\label{defn-cellular}
Suppose $A$ is a finite-dimensional $\K$-algebra. 
A \textit{cell datum} is an ordered quadruple 
$(\P,\Tt,\Ct,\I)$, where 
$(\P,\leq)$ is a finite poset, 
$\Ttt$ is a finite set for all 
$\lambda \in\P$, $\I$ is a $\K$-linear 
anti-involution of $A$ and $\Ct$ is an injection
\[
\Ct\colon\coprod_{\lambda\in\P}\Ttt\times \Ttt\to A,\;(i,j)\mapsto c^{\lambda}_{ij}.
\]
The whole data should be such that the $c^{\lambda}_{ij}$'s form a basis of 
$A$ with $\I(c^{\lambda}_{ij})=c^{\lambda}_{ji}$ 
for all $\lambda\in\P$ and all $i,j\in\Ttt$. 
Moreover, for all $a\in A$ and all $\lambda\in\P$ we have
\begin{align}\label{eq-cell1}
ac^{\lambda}_{ij}=\sum_{k\in\Ttt}r_{ik}(a)c^{\lambda}_{kj}\;(\mathrm{mod}\;A^{<\lambda})\quad\text{for all }i,j\in\Ttt.
\end{align}
Here $A^{<\lambda}$ is the subspace of $A$ spanned by the set 
$\{c^{\mu}_{ij}\mid \mu<\lambda\text{ and }i,j\in\Tt(\mu)\}$ and the 
scalars $r_{ik}(a)\in\K$ 
are supposed to be independent of $j$.

An algebra $A$ with such a quadruple is called a \textit{cellular algebra} and 
the $c^{\lambda}_{ij}$ are called a \textit{cellular basis} of $A$ 
(with respect to the $\K$-linear 
anti-involution $\I$).
\end{defn}

Let us fix $T\in\T$ in the following. We will now construct
cellular bases of $\Endrr$ in the semisimple as well as in the non-semisimple case.

To this end, we need to specify the cell datum. Set
\[
(\P,\leq)=(\{\lambda\in X^+\mid (T:\Nl)=(T:\Dl)\neq 0\},\leq),
\]
where $\leq$ is the usual partial ordering on $X^+$, see 
at the beginning of \fullref{sub-quantumgroups}. 
Note that $\P$ is finite since $T$ is finite-dimensional.
Moreover, motivated by \fullref{thm-baseshom2}, for each $\lambda\in\P$ define
$\Ttt=\{1,\dots,(T:\Nl)\}=\{1,\dots,(T:\Dl)\}=\cal{J}^{\lambda}$.

Recalling that we write $\mathrm{i}(\cdot)=\cal{D}(\cdot)$ (for 
$\cal{D}$ being the duality functor 
from \fullref{sub-quantumgroups} 
that exchanges Weyl and dual Weyl $\Uq$-modules 
and fixes all $\Uq$-tilting modules), 
the assignment
$\I\colon\Endrr\to\Endrr,\phi\mapsto\cal{D}(\phi)$
is clearly a $\K$-linear anti-involution. 
Choose any basis 
$G^{\lambda}$ 
of $\Hom_{\Uq}(\Dl,T)$ as above and 
any lifts $\overline{g}_i^{\lambda}$. Then $\I(G^{\lambda})$ is a basis of 
$\Hom_{\Uq}(T,\Nl)$ and $\I(\overline{g}_i^{\lambda})$ is a 
lift of $\I(g_i^{\lambda})$. 
By \fullref{cor-baseshom2} we see that
\[
\{
c_{ij}^{\lambda}=\overline{g}_i^{\lambda}\circ\mathrm{i}(\overline{g}_j^{\lambda})
=\overline{g}_i^{\lambda}\circ \overline{f}_j^{\lambda}\mid \lambda\in\P,\; i,j\in \Ttt
\}
\]
is a basis of $\Endrr$. 
Finally let $\Ct\colon\Ttt\times\Ttt\to\Endrr$ be given by $(i,j)\mapsto c_{ij}^{\lambda}$.

Now we are ready to state and prove our main theorem.

\begin{thm}(\textbf{A cellular basis for $\Endrr$.})\label{thm-cell}
The quadruple $(\P,\Tt,\Ct,\I)$ defined above is 
a cell datum for $\Endrr$.\makeqed
\end{thm}

\begin{proof}
As mentioned above, the sets $\P$ and $\Ttt$ are 
finite for all $\lambda\in\P$. Moreover, $\I$ is a $\K$-linear anti-involution of $\Endrr$ 
and the $c_{ij}^{\lambda}$'s form a 
basis of $\Endrr$ by \fullref{cor-baseshom2}.
Because the functor $\cal{D}(\cdot)$ is contravariant, we see that
\[
\mathrm{i}(c_{ij}^{\lambda})=\I(\overline{g}_i^{\lambda}\circ\I(\overline{g}_j^{\lambda}))
=\overline{g}_j^{\lambda}\circ\I(\overline{g}_i^{\lambda})=c_{ji}^{\lambda}.
\]

Thus, only the condition \eqref{eq-cell1} remains to be proven.
For this purpose, let $\varphi\in\Endrr$. Since $\varphi\circ\overline{g}^{\lambda}_i\circ\iota^{\lambda}=\varphi\circ g^{\lambda}_i\in\Hom_{\Uq}(\Dl,T)$, we have coefficients $r^{\lambda}_{ik}(\varphi)\in\K$ such that
\begin{equation}\label{eq-action}
\varphi\circ g^{\lambda}_i=\sum_{k\in\Ttt} r^{\lambda}_{ik}(\varphi)g_k^{\lambda},
\end{equation}
because we know that the $g_i^{\lambda}$'s form a basis of $\Hom_{\Uq}(\Dl,T)$.
But this implies then that $\varphi\circ\overline{g}^{\lambda}_i-\sum_{k\in\Ttt} r^{\lambda}_{ik}(\varphi)\overline{g}_k^{\lambda}\in\Hom_{\Uq}(\Tl,T)^{<\lambda}$, so that
\[
\varphi\circ\overline{g}^{\lambda}_i\circ \overline{f}^{\lambda}_j-\sum_{k\in\Ttt} r^{\lambda}_{ik}(\varphi)\overline{g}_k^{\lambda}\circ \overline{f}^{\lambda}_j\in\Hom_{\Uq}(T,T)^{<\lambda}=\End_{\Uq}(T)^{<\lambda},
\]
which proves \eqref{eq-cell1}. The theorem follows.
\end{proof}
\section{The cellular structure and \texorpdfstring{$\Mod{\Endrr}$}{EndMod}}\label{sec-endoringsrep}
The goal of this section is to present the 
representation theory of cellular algebras 
for $\Endrr$ from the viewpoint of $\Uq$-tilting theory.  
In fact, most of the results in this section 
are not new and have been proved for general 
cellular algebras, see e.g. \cite[Section 3]{gl}. 
However, they take a nice and easy form 
in our setup. The last theorem, the 
semisimplicity criterion 
from \fullref{thm-cellsemisimple}, is new and has 
potentially many applications, see e.g. \cite{ast}.

\subsection{Cell modules for \texorpdfstring{$\Endrr$}{EndUq(T)}}\label{sub-cellmodules}

We study now the representation theory 
for $\Endrr$ via the cellular structure we have found for it. 
We denote its module 
category by $\Mod{\Endrr}$.

\begin{defn}(\textbf{Cell modules.})\label{defn-cellmod}
Let $\lambda\in\P$. The \textit{cell module} associated to $\lambda$ is the 
left $\Endrr$-module given by 
$
C(\lambda)=\Hom_{\Uq}(\Dl,T).
$
The right $\Endrr$-module given by 
$
C(\lambda)^*=\Hom_{\Uq}(T,\Nl)
$
is called the \textit{dual cell module} associated to $\lambda$.
\end{defn}

The link to the definition of cell modules from \cite[Definition 2.1]{gl} is 
given via our choice of basis
$\{g_i^{\lambda}\}_{i\in\Ttt}$. In this 
basis the action of $\Endrr$ on $C(\lambda)$ is given by
\begin{equation}\label{eq-action1}
\varphi\circ g^{\lambda}_i=\sum_{k\in\Ttt} r^{\lambda}_{ik}(\varphi)g_k^{\lambda},\quad \varphi\in\Endrr,
\end{equation}
see \eqref{eq-action}. Here the coefficients are the 
same as those appearing when we consider the left action of $\Endrr$ on 
itself in terms of the cellular basis 
$\{c_{ij}^{\lambda}\}^{\lambda\in\P}_{i,j\in\Ttt}$, that is,
\begin{equation}\label{eq-action2}
\varphi\circ c^{\lambda}_{ij}=\sum_{k\in\Ttt} r^{\lambda}_{ik}(\varphi)c_{kj}^{\lambda}\;(\mathrm{mod}\;\Endrr^{<\lambda}),\quad \varphi\in\Endrr.
\end{equation}

In a completely similar fashion: the dual cell module $C(\lambda)^*$ has a basis 
consisting of $\{f_j^{\lambda}\}_{j\in\Ttt}$ with 
$f_j^{\lambda}=\I(g_j^{\lambda})$. In this basis the right action of 
$\Endrr$ is given via
\begin{equation}\label{eq-action3}
f^{\lambda}_j\circ \varphi=\sum_{k\in\Ttt} r^{\lambda}_{kj}(\I(\varphi))f_k^{\lambda},\quad \varphi\in\Endrr.
\end{equation}

We can use the unique $\Uq$-homomorphism from \eqref{eq-uniquemap} 
and the duality functor $\cal{D}(\cdot)$
to define the following \textit{cellular pairing} in the spirit of 
Graham and Lehrer \cite[Definition 2.3]{gl}.

\begin{defn}(\textbf{Cellular pairing.})\label{defn-cellpair}
Let $\lambda\in\P$. Then we denote by $\vartheta^{\lambda}$ the $\K$-bilinear form 
$\vartheta^{\lambda}\colon C(\lambda)\otimes C(\lambda)\to \K$ determined by the property
\[
\I(h)\circ g=\vartheta^{\lambda}(g,h)c^{\lambda},\quad g,h\in C(\lambda)=\Hom_{\Uq}(\Dl,T).
\]
We call $\vartheta^{\lambda}$ the \textit{cellular pairing} associated to $\lambda\in\P$.
\end{defn}

\begin{lem}\label{lem-thethird}
The cellular pairing 
$\vartheta^{\lambda}$ is well-defined, symmetric and contravariant.\makeqed
\end{lem}

\begin{proof}
That $\vartheta^{\lambda}$ is well-defined follows directly 
from the uniqueness of $c^{\lambda}$.

Applying $\I$ to the defining equation of $\vartheta^{\lambda}$ gives
\[
\vartheta^{\lambda}(g,h)\I(c^{\lambda})=\I(\vartheta^{\lambda}(g,h)c^{\lambda})=\I(\I(h)\circ g)=\I(g)\circ h=\vartheta^{\lambda}(h,g)c^{\lambda},
\]
and thus, $\vartheta^{\lambda}(g,h)=\vartheta^{\lambda}(h,g)$, 
because $c^{\lambda}=\I(c^{\lambda})$. 
(Recall that 
$c^{\lambda}\colon\Dl\to\Nl$ is unique up to scalars. 
Hence, we can fix scalars accordingly such 
that $c^{\lambda}=\I(c^{\lambda})$.) Similarly, 
contravariance of $\cal{D}(\cdot)$ gives
\[
\vartheta^{\lambda}(\varphi\circ g,h)=\vartheta^{\lambda}(g,\I(\varphi)\circ h),\quad\varphi\in\Endrr,\; g,h\in C(\lambda),
\]
which shows contravariance of the cellular pairing.
\end{proof}

\begin{prop}\label{prop-thefirst}
Let $\lambda\in\P$. Then $\Tl$ is a summand of $T$ if and only if $\vartheta^{\lambda}\neq 0$.\makeqed
\end{prop}

\begin{proof}
(See also \cite[Proposition 1.5]{an0}.)
Assume $T\cong \Tl\oplus \text{rest}$. We denote by 
$\overline{g}\colon \Tl\to T$ and by $\overline{f}\colon T\to\Tl$ the 
corresponding inclusion and projection respectively. As usual, set 
$g=\overline{g}\circ\iota^{\lambda}$ and $f=\pi^{\lambda}\circ\overline{f}$. Then we have
$
f\circ g\colon \Dl\hookrightarrow \Tl\hookrightarrow 
T\twoheadrightarrow\Tl\twoheadrightarrow\Nl=c^{\lambda}
$ (mapping head to socle),
giving $\vartheta^{\lambda}(g,\I(f))=1$. This shows that $\vartheta^{\lambda}\neq 0$.

Conversely, assume that there exist $g,h\in C(\lambda)$ with $\vartheta^{\lambda}(g,h)\neq 0$. 
Then the commuting ``bow tie diagram'', i.e.
\[
\begin{gathered}
\begin{xy}
  \xymatrix{
  	    \Dl \ar[rd]^{g}\ar@{^{(}->}[d]_{\iota^{\lambda}}  &    &  \\
      \Tl  \ar@{.>}[r]_{\overline{g}}    &   T \ar@{.>}[r]^/-0.5em/{\overline{\I(h)}}\ar[rd]_{\I(h)} &  \Tl\ar@{->>}[d]^/-.35em/{\pi^{\lambda}},\\
                   &    & \Nl
  }
\end{xy}
\end{gathered}
\]
shows that $\overline{\I(h)}\circ \overline{g}$ is non-zero on the 
$\lambda$-weight space of $\Tl$, because 
$\I(h)\circ g=\vartheta^{\lambda}(g,h)c^{\lambda}$. Thus, $\overline{\I(h)}\circ \overline{g}$ 
must be an isomorphism (because $\Tl$ is indecomposable and has therefore 
only invertible or nilpotent elements in $\End_{\Uq}(\Tl)$) 
showing that $T\cong \Tl\oplus \text{rest}$.
\end{proof}

In view of \fullref{prop-thefirst}, it makes sense to study the set
\begin{equation}\label{eq-po}
\P_0=\{\lambda\in \P\mid \vartheta^{\lambda}\neq 0\}\subset\P.
\end{equation}
Hence, if $\lambda\in\P_0$, then we have 
$T\cong \Tl\oplus \text{rest}$ for some $\Uq$-tilting module 
called $\text{rest}$. Note also that $\Endrr$ is quasi-hereditary if and only if $\P=\P_0$, 
see e.g. \cite[Remark 3.10]{gl}.

\subsection{The structure of \texorpdfstring{$\Endrr$}{EndUq(T)} and its cell modules}\label{sub-boring}

Recall that, for any $\lambda\in\P$, we have that 
$\Endrr^{\leq\lambda}$ and $\Endrr^{<\lambda}$ are 
two-sided ideals in $\Endrr$ (this follows 
from \eqref{eq-cell1} and its right-handed version obtained 
by applying $\I$), 
as in any cellular algebra. 
In our case we can also see this 
as follows. If $\varphi\in\Endrr^{\leq\lambda}$, then $\varphi_{\mu}=0$ unless 
$\mu\leq\lambda$. Hence, for any $\varphi,\psi\in\Endrr$ we have
$
(\varphi\circ\psi)_{\mu}=\varphi_{\mu}
\circ\psi_{\mu}=0=\psi_{\mu}\circ\varphi_{\mu}=(\psi\circ\varphi)_{\mu}
$ unless $\mu\leq\lambda$.
As a consequence, $\Endrr^{\lambda}=\Endrr^{\leq\lambda}/\Endrr^{<\lambda}$ is an $\Endrr$-bimodule.

Recall that, for any $g\in C(\lambda)$ and any $f\in C(\lambda)^*$, we denote by 
$\overline{g}\colon \Tl\to T$ and $\overline{f}\colon T\to \Tl$ a choice of lifts 
which satisfy $\overline{g}\circ \iota^{\lambda}=g$ and 
$\pi^{\lambda}\circ \overline{f}=f$, respectively.

\begin{lem}\label{lem-thefirst}
Let $\lambda\in\P$. Then the pairing map
\[
\langle\cdot,\cdot\rangle^{\lambda}\colon C(\lambda)\otimes C(\lambda)^*\to \Endrr^{\lambda},\quad
\langle g,f\rangle^{\lambda}= \overline{g}\circ\overline{f}+\Endrr^{<\lambda},
\]
with $g\in C(\lambda),f\in C(\lambda)^*$,
is an isomorphism of $\Endrr$-bimodules.\makeqed
\end{lem}

\begin{proof}
First we note that $\overline{g}\circ\overline{f}+\Endrr^{<\lambda}$ does not 
depend on the choices for the lifts 
$\overline{f},\overline{g}$, because the change-of-basis matrix from 
\fullref{lem-cell4} is unitriangular (and works 
for swapped roles of $f$'s and $g$'s as well). 
This makes the pairing well-defined.

Note that the pairing $\langle\cdot,\cdot\rangle^{\lambda}$ takes, by birth, 
the basis $(g_i^{\lambda}\otimes f_j^{\lambda})_{i,j\in\Ttt}$ of 
$C(\lambda)\otimes C(\lambda)^*$ to the 
basis $\{c_{ij}^{\lambda}+\Endrr^{<\lambda}\}_{i,j\in\Ttt}$ 
of $\Endrr^{\lambda}$ (where the latter is a basis by \fullref{lem-cell3}).

So we only need to check that
$
\langle \varphi\circ g_i^{\lambda},f_j^{\lambda}\circ \psi\rangle^{\lambda}=\varphi\circ c_{ij}^{\lambda}\circ\psi \;(\mathrm{mod}\;\Endrr^{<\lambda})
$ for any $\varphi,\psi\in\Endrr$.
But this is a direct consequence of \eqref{eq-action1}, \eqref{eq-action2} and \eqref{eq-action3}.
\end{proof}

The next 
lemma is straightforward by \fullref{lem-thefirst}. 
Details are left to the reader.

\begin{lem}\label{lem-thesecond}
We have the following.
\begin{enumerate}[label=(\alph*)]

\item There is an isomorphism of $\K$-vector spaces 
$\Endrr\cong\bigoplus_{\lambda\in\P}\Endrr^{\lambda}$.

\item \label{lem-thesecond-b} If $\varphi\in\Endrr^{\leq\lambda}$, then we have $r_{ik}^{\mu}(\varphi)=0$ 
for all $\mu\not\leq\lambda,i,k\in\Tt(\mu)$. Equivalently, 
$\Endrr^{\leq\lambda}C(\mu)=0$ unless $\mu\leq\lambda$.\qedmake

\end{enumerate}
\end{lem}

In the following we assume that $\lambda\in\P_0$ as in \eqref{eq-po}. Define 
$m_{\lambda}$ via
\begin{equation}\label{eq-love}
T\cong \Tl^{\oplus m_{\lambda}}\oplus T^{\prime},
\end{equation}
where $T^{\prime}$ is a $\Uq$-tilting module containing no summands isomorphic to $\Tl$.

Choose now a basis of $C(\lambda)=\Hom_{\Uq}(\Dl,T)$ as follows. 
For 
$i=1,\dots,m_{\lambda}$ let $\overline{g}_i^{\lambda}$ 
be the inclusion of $\Tl$ into the $i$-th 
summand of $\Tl^{\oplus m_{\lambda}}$ and set 
$g_i^{\lambda}=\overline{g}_i^{\lambda}\circ\iota^{\lambda}$.
Then extend $\{g_1^{\lambda},\dots,g_{m_{\lambda}}^{\lambda}\}$
to a basis of the cell module $C(\lambda)$ by adding 
an arbitrary basis of 
$\Hom_{\Uq}(\Dl,T^{\prime})$. Thus, in our usual notation, we have 
$c_{ij}^{\lambda}=\overline{g}_i^{\lambda}\circ \overline{f}_j^{\lambda}$ with $\overline{f}_j^{\lambda}=\I(\overline{g}_j^{\lambda})$.

In particular, 
$\overline{f}_j^{\lambda}$ 
projects onto the $j$-th summand 
in $\Tl^{\oplus m_{\lambda}}$ for $j=1,\dots, m_{\lambda}$. Thus, 
the $c^{\lambda}_{ii}$'s for $i \leq m_\lambda$ are 
idempotents in $\Endrr$ corresponding to the $i$-th 
summand in $\Tl^{\oplus m_{\lambda}}$. 
Since $\lambda\in\P_0$ (which implies $1 \leq m_\lambda$), 
$c^{\lambda}_{11}$ is always such an idempotent.
This is crucial for the following lemma, which will play an important role 
in the proof of \fullref{prop-lastone}.

\begin{lem}\label{lem-anotherone}
In the above notation:
\begin{enumerate}[label=(\alph*)]

\item \label{lem-anotherone-a} $c_{i1}^{\lambda}\circ g_1^{\lambda}=g_i^{\lambda}$ 
for all $i\in\Ttt$,

\item $c_{ij}^{\lambda}\circ g_1^{\lambda}=0$ 
for all $i,j\in\Ttt$ with $j\neq 1$.\makeqed

\end{enumerate}
\end{lem}

\begin{proof}
We have $\overline{f}_1^{\lambda}\circ g_1^{\lambda}=
\overline{f}_1^{\lambda}\circ \overline{g}_1^{\lambda}\circ\iota^{\lambda}=\iota^{\lambda}$. 
This implies $c_{i1}^{\lambda}\circ g_1^{\lambda}=
\overline{g}_i^{\lambda}\circ\iota^{\lambda}=g_i^{\lambda}$.
Next, if $j\neq 1$, then $\overline{f}_j^{\lambda}\circ g_1^{\lambda}=0$, since
$\overline{f}_j^{\lambda}$ is zero on $\Tl$. Thus, 
$c_{ij}^{\lambda}\circ g_1^{\lambda}=0$ 
for all $i,j\in\Ttt$ with $j\neq 1$.
\end{proof}

\begin{prop}(\textbf{Homomorphism criterion.})\label{prop-lastone}
Let $\lambda\in\P_0$ and fix $M\in\Mod{\Endrr}$. Then there is an isomorphism of $\K$-vector spaces
\begin{gather}\label{eq:hom-criterion}
\Hom_{\Endrr}(C(\lambda),M)\cong\{m\in M\mid \Endrr^{<\lambda}m=0\text{ and }c_{11}^{\lambda}m=m\}.
\hspace{1.6cm}
\makeqed
\hspace{-1.6cm}
\end{gather}
\end{prop}

\begin{proof}
Let $\psi\in\Hom_{\Endrr}(C(\lambda),M)$. Then 
$\psi(g_1^{\lambda})$ belongs to the right-hand side, 
because, by \fullref{lem-thesecond-b} of 
\fullref{lem-thesecond}, we get $\Endrr^{<\lambda}C(\lambda)=0$,
and we have $c_{11}^{\lambda}\circ g_1^{\lambda}=g_1^{\lambda}$ 
by \fullref{lem-anotherone-a} of 
\fullref{lem-anotherone}.
Conversely, if $m\in M$ belongs to the right-hand side in \eqref{eq:hom-criterion}, 
then we may define 
$\psi\in\Hom_{\Endrr}(C(\lambda),M)$ by $\psi(g_i^{\lambda})=c_{i1}^{\lambda}m$, $i\in\cal{I}^{\lambda}$.
Moreover, the fact that this definition gives an $\Endrr$-homomorphism follows 
from \eqref{eq-action}, \eqref{eq-action1} and \eqref{eq-action2} via 
a direct computation, 
since $\Endrr^{<\lambda}m=0$.
Clearly these two operations are mutually inverse.
\end{proof}

\begin{cor}\label{cor-okanotherone}
Let $\lambda\in\P_0$. Then $C(\lambda)$ has a unique 
simple head, denoted by $L(\lambda)$.\makeqed
\end{cor}

\begin{proof}
Set
$
\mathrm{Rad}(\lambda)=\{g\in C(\lambda)\mid \vartheta^{\lambda}(g,C(\lambda))=0\}.
$
As the cellular pairing $\vartheta^{\lambda}$ from 
\fullref{defn-cellpair} is contravariant 
by \fullref{lem-thethird}, we see that $\mathrm{Rad}(\lambda)$ is an 
$\Endrr$-submodule of $C(\lambda)$. Since $\vartheta^{\lambda}\neq 0$ for $\lambda\in\P_0$, we 
have $\mathrm{Rad}(\lambda)\subsetneq C(\lambda)$. We claim that $\mathrm{Rad}(\lambda)$ is the 
unique maximal proper $\Endrr$-submodule of $C(\lambda)$.

Let $g\in C(\lambda)-\mathrm{Rad}(\lambda)$. Moreover, choose $h\in C(\lambda)$ with 
$\vartheta^{\lambda}(g,h)=1$. Then $\I(h)\circ g=c^{\lambda}$ so that 
$\overline{\I(h)}\circ g=\iota^{\lambda}\;(\mathrm{mod}\;\Endrr^{<\lambda})$.
Therefore, 
\[
g^{\prime}=\overline{g}^{\prime}\circ\overline{\I(h)}\circ g \;(\mathrm{mod}\;\Endrr^{<\lambda}),
\quad\text{for all }g^{\prime}\in C(\lambda).
\] 
This implies $C(\lambda)=\Endrr^{\leq\lambda}g$. Thus, any proper 
$\Endrr$-submodule of $C(\lambda)$ is contained in $\mathrm{Rad}(\lambda)$ which implies 
the desired statement.
\end{proof}

\begin{cor}\label{cor-okanotheronenext}
Let $\lambda\in\P_0,\mu\in\P$ and assume that 
$\Hom_{\Endrr}(C(\lambda),M)\neq 0$ for some 
$\Endrr$-module $M$ isomorphic to a subquotient of $C(\mu)$. 
Then we have $\mu\leq \lambda$. In particular, all composition factors 
$L(\lambda)$ of $C(\mu)$ satisfy $\mu\leq \lambda$.\makeqed
\end{cor}

\begin{proof}
By \fullref{prop-lastone} the assumption in the 
corollary implies the existence of an element 
$m\in M$ with $c^{\lambda}_{11}m=m$. But if $\mu\not\leq \lambda$, then 
$c^{\lambda}_{11}$ vanishes on the $\Uq$-weight space $T_{\mu}$ and hence, $c^{\lambda}_{11}g$ 
kills the highest weight vector in $\Delta_q(\mu)$ for all $g\in C(\mu)$. This makes the existence of
such an $m\in M$ impossible 
unless $\mu\leq \lambda$.
\end{proof}

\subsection{Simple \texorpdfstring{$\Endrr$}{EndUq(T)}-modules and semisimplicity of \texorpdfstring{$\Endrr$}{EndUq(T)}}\label{sub-cellmodules2}

Let $\lambda\in\P_0$. Note that \fullref{cor-okanotherone} shows that 
$C(\lambda)$ has a unique simple head $L(\lambda)$. 
We then arrive at the following classification of all simple 
modules in $\Mod{\Endrr}$.

\begin{thm}(\textbf{Classification of simple $\Endrr$-modules.})\label{thm-cellclass}
The set
$
\{L(\lambda)\mid \lambda\in\P_0\}
$
forms a complete set of pairwise non-isomorphic, simple $\Endrr$-modules.\makeqed
\end{thm}

\begin{proof}
We have to show three statements, namely that the $L(\lambda)$'s are 
simple, that they are pairwise non-isomorphic and that every 
simple $\Endrr$-module is one of the $L(\lambda)$'s.

Because the first statement follows directly from the definition of $L(\lambda)$ 
(see \fullref{cor-okanotherone}), 
we start by showing the second. Thus, assume that $L(\lambda)\cong L(\mu)$ for 
some $\lambda,\mu\in\P_0$. Then
\[
\Hom_{\Endrr}(C(\lambda),C(\mu)/\mathrm{Rad}(\mu))\neq 0\neq \Hom_{\Endrr}(C(\mu),C(\lambda)/\mathrm{Rad}(\lambda)).
\]
By \fullref{cor-okanotheronenext}, we get $\mu\leq \lambda$ 
and $\lambda\leq \mu$ from the left and right-hand side. Thus, $\lambda=\mu$.

Suppose that $L\in\Mod{\Endrr}$ is simple. Then we can choose $\lambda\in \P$ minimal such 
that (recall that $\Endrr^{\leq\lambda}$ is a two-sided ideal)
\begin{equation}\label{eq-something}
\Endrr^{<\lambda}L=0\quad\text{and}\quad \Endrr^{\leq\lambda}L=L.
\end{equation}
We claim that $\lambda\in\P_0$. Indeed, if not, then, by \fullref{prop-thefirst}, 
we see that $\Tl$ is not a summand of $T$.
Hence, in our usual notation, all $\overline{f}^{\lambda}_j\circ\overline{g}^{\lambda}_{i^{\prime}}$ 
vanishes on the $\lambda$-weight space. 
It follows that $c_{ij}^{\lambda}c_{i^{\prime}j^{\prime}}^{\lambda}$ also 
vanish on the $\lambda$-weight space for all $i,j,i^{\prime},j^{\prime}\in\Ttt$. 
This means that we have
$\Endrr^{\leq\lambda}\Endrr^{\leq\lambda}\subset \Endrr^{<\lambda}$ 
making \eqref{eq-something} impossible.

For $\lambda\in\P_0$ we see by \fullref{lem-anotherone} that
\begin{equation}\label{eq-problemsolved}
c_{i1}^{\lambda}c_{1j}^{\lambda}=c_{ij}^{\lambda} \;(\mathrm{mod}\;\Endrr^{<\lambda}).
\end{equation}
Hence, by \eqref{eq-something}, there exist $i,j\in\Ttt$ 
such that $c_{ij}^{\lambda}L\neq 0$.
By \eqref{eq-problemsolved} we also have 
$c_{i1}^{\lambda}L\neq 0\neq c_{1j}^{\lambda}L$. This in turn (again by \eqref{eq-problemsolved}) 
ensures that $c_{11}^{\lambda}L\neq 0$. Take then 
$m\in c_{11}^{\lambda}L-\{0\}$ and observe that $c_{11}^{\lambda}m=m$. Hence, 
by \fullref{prop-lastone}, there is a non-zero $\Endrr$-homomorphism $C(\lambda)\to L$. The conclusion follows now from \fullref{cor-okanotherone}.
\end{proof}

Recall from \fullref{sub-boring} the notation $m_{\lambda}$ 
(the multiplicity of $\Tl$ in $T$) 
and the choice of 
basis for $C(\lambda)$ (in the paragraphs before 
\fullref{lem-anotherone}). Then we get the following 
connection between the decomposition of $T$ as in \eqref{eq-love} 
and the simple $\Endrr$-modules $L(\lambda)$.

\begin{thm}(\textbf{Dimension formula.})\label{thm-dim}
If $\lambda\in\P_0$, then $\dim(L(\lambda))=m_{\lambda}$.\makeqed
\end{thm}

Note that this result is implicit in \cite{gl} 
and has also been observed in e.g. \cite{erd} and \cite{soe5}.

\begin{proof}
We use the notation from \fullref{sub-boring}. Since $T^{\prime}$ has no summands isomorphic 
to $\Tl$, we see that $\Hom_{\Uq}(\Dl,T^{\prime})\subset \mathrm{Rad}(\lambda)$ (see the proof of 
\fullref{cor-okanotherone}).
On the other hand, $g_i^{\lambda}\notin\mathrm{Rad}(\lambda)$ for $1\leq i\leq m_{\lambda}$ because 
for these $i$ we have $f_i^{\lambda}\circ g_i^{\lambda}=c^{\lambda}$ by construction.
Thus, the statement follows.
\end{proof}

\begin{thm}(\textbf{Semisimplicity criterion.})\label{thm-cellsemisimple}
The cellular algebra $\Endrr$ is semisimple if and only if $T$ is a semisimple $\Uq$-module.\makeqed
\end{thm}
\begin{proof}
Note that the $\Tl$'s are simple if and only if $\Tl\cong\Dl\cong\Ll\cong\Nl$.
Hence, $T$ is semisimple as a $\Uq$-module if and only if 
$T=\bigoplus_{\lambda\in\P_0}\Dl^{\oplus m_{\lambda}}$ 
with $m_{\lambda}$ as in \fullref{sub-boring}.

Thus, we see that, if $T$ decomposes into simple 
$\Uq$-modules, then $\Endrr$ is semisimple by the Artin--Wedderburn theorem (since 
$\Endrr$ will decompose into a direct sum of matrix algebras in this case).

On the other hand, if $\Endrr$ is semisimple, then we know, by \fullref{cor-okanotherone}, 
that the cell 
modules $C(\lambda)$ are simple, i.e. $C(\lambda)=L(\lambda)$ 
for all $\lambda\in\P_0$. Then
\begin{equation}\label{eq-needed}
T\cong\bigoplus_{\lambda\in\P_0}\Tl^{\oplus m_{\lambda}},\quad m_{\lambda}=\dim(L(\lambda))=\dim(C(\lambda))=\dim(\Hom_{\Uq}(\Dl,T))
\end{equation}
by \fullref{thm-dim}. Assume now 
that there exists a summand $T_q(\lambda^{\prime})$ of 
$T$ as in \eqref{eq-needed} with 
$T_q(\lambda^{\prime})\not\cong\Delta_q(\lambda^{\prime})$ and choose 
$\lambda^{\prime}\in\P_0$ minimal with this property.

Then there exists a $\mu<\lambda^{\prime}$ such that 
$\Hom_{\Uq}(\Delta_q(\mu),T_q(\lambda^{\prime}))\neq 0$. Choose also $\mu$ minimal 
among those. By our usual 
construction this then gives in turn a non-zero $\Uq$-homomorphism 
$\overline{g}\circ\overline{f}\colon T_q(\lambda^{\prime})\to T_q(\mu)\to T_q(\lambda^{\prime})$. 
By \eqref{eq-needed}, we can extend 
$\overline{g}\circ\overline{f}$ to an element of $\Endrr$ by 
defining it to be zero on all other summands.

Clearly, by construction, 
$(\overline{g}\circ\overline{f})C(\mu^{\prime})=0$ for $\mu^{\prime}\in\P_0$ 
with $\mu^{\prime}\neq\lambda^{\prime}$ and $\mu^{\prime}\not\leq\mu$. If 
$\mu^{\prime}\leq\mu$, then consider $\varphi\in C(\mu^{\prime})$. Then 
$(\overline{g}\circ\overline{f})\circ\varphi=0$ unless $\varphi$ has some non-zero component 
$\varphi^{\prime}\colon\Delta_q(\mu^{\prime})\to T_q(\lambda^{\prime})$. This 
forces $\mu^{\prime}=\mu$ by minimality of $\mu$. But since 
$\Delta_q(\mu^{\prime})\cong T_q(\mu^{\prime})$, by minimality of $\lambda^{\prime}$, we 
conclude that $\overline{f}\circ\varphi=0$ (otherwise 
$T_q(\mu^{\prime})$ would be a summand of $T_q(\lambda^{\prime})$).

Hence, the non-zero element $\overline{g}\circ\overline{f}\in\Endrr$ 
kills all $C(\mu^{\prime})$ for $\mu^{\prime}\in\P_0$. This contradicts the semisimplicity 
of $\Endrr$: as noted above, $C(\lambda)=L(\lambda)$ for all $\lambda\in\P_0$ which 
implies $\Endrr\cong\bigoplus_{\lambda\in\P_0}C(\lambda)^{\oplus k_{\lambda}}$ 
for some $k_{\lambda}\in\N$.
\end{proof}
\section{Cellular structures: examples and applications}\label{sec-examples}
In this section we provide many examples of cellular 
algebras arising from our main theorem. This includes several 
renowned examples where cellularity is known (but usually proved case by case spread 
over the literature 
and with cellular bases which differ in general from ours), and also new ones.
In the first subsection we give a full treatment of the semisimple case and 
describe how to obtain all the examples from the introduction using our methods.
In the second subsection we focus on the Temperley--Lieb 
algebras $\cal{TL}_d(\delta)$ and 
give a detailed account how to apply our results to these.

\subsection{Cellular structures using \texorpdfstring{$\Uq$}{Uq}-tilting modules: several examples}\label{sub-examples}

In the following let $\omega_i$ for $i=1,\dots,n$ denote the fundamental weights 
(of the corresponding type).

\subsubsection{The semisimple case}\label{subsub-semi}

Suppose the category $\Mod{\Uq}$ is semisimple, that is, 
$q$ is not a root of unity in $\K^{\ast}-\{1\}$ or $q=\pm 1\in\K$ with 
$\Char(\K)=0$.

In this case $\T=\Mod{\Uq}$ and any $T\in\T$ has a
decomposition $T\cong \bigoplus_{\lambda\in X^+}\Dl^{\oplus m_{\lambda}}$ with the multiplicities
$m_{\lambda}=(T:\Dl)$. This induces an Artin--Wedderburn decomposition
\begin{equation}\label{eq-artinwed}
\Endrr\cong \bigoplus_{\lambda\in X^+}M_{m_{\lambda}}(\K)
\end{equation}
into matrix algebras.
A natural choice of basis for $\Hom_{\Uq}(\Dl,T)$ is
\[
G^{\lambda}=\{g_1^{\lambda},\dots,g_{m_{\lambda}}^{\lambda}\mid g_i^{\lambda}\colon\Dl\hookrightarrow T\text{ is the inclusion into the i-th summand}\}.
\]
Then our cellular basis consisting of the 
$c_{ij}^{\lambda}$'s as 
in \fullref{sub-celltilting} 
(no lifting is needed in this case) is an 
Artin--Wedderburn basis, i.e.,
a basis that realizes the decomposition 
\eqref{eq-artinwed} in the following sense. 
The basis element $c_{ij}^{\lambda}$ is the matrix 
$\mathbf{E}_{ij}^{\lambda}$ (in the $\lambda$-summand on 
the right-hand side in \eqref{eq-artinwed}) which has all 
entries zero except one 
entry equals $1$ in the $i$-th row and $j$-th column.
Note that, as expected in this case, 
$\Endrr$ has, by the \fullref{thm-cellclass} 
and \fullref{thm-dim}, one simple $\Endrr$-module 
$L(\lambda)$ of dimension $m_{\lambda}$ for all 
summands $\Dl$ of $T$.

\subsubsection{The symmetric group and the Iwahori--Hecke algebra}\label{subsub-hecke}

Let us fix $d\in \Z_{\geq 0}$ and let 
us denote by $S_d$ the symmetric group in $d$ letters and by
$\cal{H}_d(q)$ its associated Iwahori--Hecke algebra. 
We note that $\K[S_d]\cong\cal{H}_d(1)$. 
Moreover, let $\Uq=\Uq(\gll{n})$. The vector representation of $\Uq$, 
which we denote by $V=\K^n=\Delta_q(\omega_1)$, is a $\Uq$-tilting module 
(since $\omega_1$ is minimal in $X^+$). 
Set $T=V^{\otimes d}$, which is again a $\Uq$-tilting module. Quantum Schur--Weyl duality 
(see \cite[Theorem 6.3]{dps} for surjectivity and 
use Ext-vanishing for the fact that $\dim(\Endrr)$ 
is obtained via base change from $\Z[v,v^{-1}]$ to $\K$ 
for all $\K$ and $q\in\K^{\ast}$) states that
\begin{equation}\label{eq-schurweyl3}
\Phi_{q\mathrm{SW}}\colon\cal{H}_d(q)\twoheadrightarrow
\Endrr\quad\text{and}\quad\Phi_{q\mathrm{SW}}\colon
\cal{H}_d(q)\xrightarrow{\cong}\Endrr,\text{ if }n\geq d.
\end{equation}
Thus, our main result implies that $\cal{H}_d(q)$, and in particular 
$\K[S_d]$, are cellular for any $q\in\K^{\ast}$ and any field $\K$ (by taking $n\geq d$).

In this case the cell modules for $\Endrr$ are usually called
Specht modules $S^{\lambda}_{\K}$ and our
\fullref{thm-dim} gives the following.
\begin{itemize}
\item If $q=1$ and $\Char(\K)=0$, then the dimension 
$\dim(S^{\lambda}_{\K})$ is equal to the multiplicity 
of the simple $\Uu_1$-module 
$\Delta_1(\lambda)\cong L_1(\lambda)$ in $V^{\otimes d}$ for all $\lambda\in\P^0$. 
These numbers are given by known formulas (e.g. the hook length formula).
\item If $q=1$ and $\Char(\K)>0$, then the dimension of the simple head $D^{\lambda}_{\K}$
of $S^{\lambda}_{\K}$ is the multiplicity with which 
$T_1(\lambda)$ occurs as a summand in $V^{\otimes d}$ for all $\lambda\in\P_0$, see also \cite{erd}. 
It is a wide open problem to determine these numbers. (See however \cite{rw}.)
\item If $q$ is a complex, primitive root of unity, then we can compute the 
dimension of the simple $\cal{H}_d(q)$-modules by using the algorithm
as in \cite{astproofs}. 
In particular, this connects with the LLT algorithm from \cite{llt}.
\item If $q$ is a root of unity and $\K$ is arbitrary, 
then not much is known. Still, our methods apply 
and we get a way to calculate the dimensions of the simple $\cal{H}_d(q)$-modules, 
if we can decompose $T$ into its indecomposable summands.
\end{itemize}

\subsubsection{The Temperley--Lieb algebra and other \texorpdfstring{$\sll{2}$}{sll2}-related algebras}\label{subsub-othersl2}

Let $\Uq=\Uq(\sll{2})$ and let $T$ be as in \fullref{subsub-hecke} with $n=2$. 
For any $d\in\Z_{\geq 0}$ we have 
$\cal{TL}_d(\delta)\cong\Endrr$ by Schur--Weyl duality, 
where $\cal{TL}_d(\delta)$ is the Temperley--Lieb algebra in $d$-strands with parameter 
$\delta=q+q^{-1}$. This works for all $\K$ and 
all $q\in\K^{\ast}$ (this can be deduced from, for
example, \cite[Theorem 6.3]{dps}). 
Hence, $\cal{TL}_d(\delta)$ is always cellular.
We discuss this case in more detail in \fullref{sub-graded}.

Furthermore, if we are in the semisimple case, then $\Delta_q(i)$ is a 
$\Uq$-tilting module 
for all $i\in\Z_{\geq 0}$ and so is 
$T=\Delta_q(i_1)\otimes \cdots \otimes \Delta_q(i_d)$. Thus, we obtain that $\Endrr$
is cellular. 

The algebra $\Endrr$
is known to give a diagrammatic presentation of the (tensor) category of 
$\Uq$-modules, see \cite{rt}, and can be used to define the
colored Jones polynomial.

If $q\in\K$ is a root of unity and $l$ is the order of $q^2$, then, for any $0\leq i<l$,  
$\Delta_q(i)$ is a $\Uq$-tilting module 
(since its simple)
and so is $T=\Delta_q(i)^{\otimes d}$. The endomorphism algebra $\Endrr$ is cellular. This reproves 
parts of \cite[Theorem 1.1]{alz} using our general approach.

In characteristic $0$: Another family of $\Uq$-tilting 
modules was studied in \cite{at}. For any $d\in\Z_{\geq 0}$, fix any 
$\lambda_0\in\{0,\dots, l-2\}$ and consider
$T=T_q(\lambda_0)\oplus\dots\oplus T_q(\lambda_d)$ 
where $\lambda_k$ 
is the unique integer $\lambda_k\in\{kl,\dots, (k+1)l-2\}$ 
linked to $\lambda_0$. 
We again obtain that $\Endrr$ is cellular.
Note that $\Endrr$ 
can be identified with the so-called 
(type $A$) zig-zag algebra $A_d$, 
see \cite[Proposition 3.9]{at}, 
introduced in \cite{hk1}.
These algebras are naturally graded making $\Endrr$ into a 
graded cellular algebra in the sense of \cite{hm} and are special examples arising from 
the family of generalized Khovanov arc algebras whose cellularity is studied in \cite{bs1}.

\subsubsection{Spider algebras}\label{subsub-spid}

Let $\Uq=\Uq(\sll{n})$ (or, alternatively, $\Uq(\gll{n})$). 
One has for any $q\in\K^{\ast}$ that all $\Uq$-representations 
$\Delta_q(\omega_i)$ are $\Uq$-tilting modules 
(because the $\omega_i$'s are minimal in $X^+$). Hence, for any $k_i\in\{1,\dots,n-1\}$,
$T=\Delta_q(\omega_{k_1})\otimes \cdots \otimes \Delta_q(\omega_{k_d})$ is a 
$\Uq$-tilting module. Thus, $\Endrr$ is cellular. 
These algebras are related to type $A_{n-1}$ 
spider algebras as in \cite{kup}, are connected to the Reshetikhin--Turaev
$\sll{n}$-link polynomials and give a diagrammatic description of the representation 
theory of $\sll{n}$, see \cite{ckm}, providing a link from our 
work to low-dimensional topology and diagrammatic algebra.
Note that cellular bases 
(which, in this case, coincide with our cellular bases) 
of these were 
found in \cite[Theorem 2.57]{elias}.

More general: In any type we have that $\Delta_q(\lambda)$ are 
$\Uq(\fg)$-tilting modules for 
minuscule $\lambda\in X^+$, 
see \cite[Part II, Chapter 2, Section 15]{jarag}. Moreover, if $q$ is a 
root of unity ``of order $l$ big enough'' (ensuring that the 
$\omega_i$'s are in the closure of the fundamental alcove), then the $\Delta_q(\omega_i)$ 
are $\Uq(\fg)$-tilting modules by the linkage principle 
(see \cite[Corollaries 4.4 and 4.6]{an2}).
So in these cases we can 
generalize the above results to other types.

Still more generally: we may take (for any type and any $q\in\K^{\ast}$) 
arbitrary $\lambda_j\in X^+$ (for $j=1,\dots,d$) and obtain a cellular structure on $\Endrr$ 
for $T=T_q(\lambda_1)\otimes\cdots\otimes T_q(\lambda_d)$.

\subsubsection{The Ariki--Koike algebra and related algebras}\label{subsub-akblob}

Take $\fg=\gll{m_1}\oplus\dots\oplus\gll{m_r}$ 
(which can be easily fit into our context) with $m_1+\cdots+m_r=m$ and let $V$ be the vector 
representation of $\Uu_1(\gll{m})$ restricted to $\Uu_1=\Uu_1(\fg)$. This is again a 
$\Uu_1$-tilting module 
and so is $T=V^{\otimes d}$.
Then 
we have a cyclotomic analog of \eqref{eq-schurweyl3}, namely
\begin{equation}\label{eq-schurweyl4}
\Phi_{\mathrm{cl}}\colon\C[\Z/r\Z\wr S_d]\twoheadrightarrow\End_{\Uu_1}(T)\quad\text{and}\quad\Phi_{\mathrm{cl}}\colon
\C[\Z/r\Z\wr S_d]\xrightarrow{\cong}\End_{\Uu_1}(T),\text{ if }m\geq d,
\end{equation}
where $\C[\Z/r\Z\wr S_d]$ is the group algebra of the complex 
reflection group $\Z/r\Z\wr S_d\cong (\Z/r\Z)^d\rtimes S_d$, 
see \cite[Theorem 9]{mas}. Thus, we 
can apply our main theorem and obtain a cellular basis for these quotients of
$\C[\Z/r\Z\wr S_d]$. If $m\geq d$, then \eqref{eq-schurweyl4} is an 
isomorphism (see Lemma 11 loc. cit.) 
and we obtain that $\C[\Z/r\Z\wr S_d]$ itself is a 
cellular algebra for all $r,d$.
In the extremal case $m_1=m-1$ and $m_2=1$, 
the resulting quotient of \eqref{eq-schurweyl4} is known 
as Solomon's algebra introduced in \cite{sol} 
(also called the algebra of the inverse semigroup or the rook monoid algebra) 
and we obtain that Solomon's algebra is cellular.
In the extremal case $m_1=m_2=1$, the resulting 
quotient is a specialization of the blob algebra $\cal{BL}_d(1,2)$
(in the notation used in \cite{rh}).
To see this, note that both algebras are quotients of $\C[\Z/r\Z\wr S_d]$.
The kernel of the quotient to $\cal{BL}_d(1,2)$ is 
described explicitly by Ryom-Hansen in \cite[(1)]{rh} 
and is by \cite[Lemma 11]{mas} contained in the kernel of 
$\Phi_{\mathrm{cl}}$ from \eqref{eq-schurweyl4}. 
Since both algebras have the same dimensions, they are isomorphic.

Let $\Uq=\Uq(\fg)$. We get in the quantized case (for $q\in\C-\{0\}$ not a root of unity)
\begin{equation}\label{eq-schurweyl5}
\Phi_{q\mathrm{cl}}\colon\cal{H}_{d,r}(q)\twoheadrightarrow\End_{\Uq}(T)\quad\text{and}\quad
\Phi_{q\mathrm{cl}}\colon\cal{H}_{d,r}(q)\xrightarrow{\cong}\End_{\Uq}(T),\text{ if }m\geq d,
\end{equation}
where $\cal{H}_{d,r}(q)$ is the Ariki--Koike algebra introduced 
in \cite{ak}. A proof of \eqref{eq-schurweyl5} 
can for example be found in \cite[Theorem 4.1]{sash}. 
Thus, as before, our main theorem applies and we obtain: 
the Ariki--Koike algebra $\cal{H}_{d,r}(q)$ is cellular 
(by taking $m\geq d$), the quantized rook monoid algebra $\cal{R}_d(q)$ 
from \cite{hr} is 
cellular 
and the blob algebra $\cal{BL}_d(q,m)$ is cellular (which follows as above).
Note that the cellularity of $\cal{H}_{d,r}(q)$ was obtained in \cite{djm}, the 
cellularity of the quantum rook monoid algebras and of the 
blob algebra can be found in \cite{pag} and in \cite{rh3} respectively.

In fact, \eqref{eq-schurweyl5} is  
still true in the non-semisimple cases, 
see \cite[Theorem 1.10 and Lemma 2.12]{hs} as long as $\K$ satisfies 
a certain separation condition (which implies that the algebra in question has 
the right dimension, see \cite{ari}). Again, our main theorem 
applies.

\subsubsection{The Brauer algebras and related algebras}\label{subsub-brauer}

Consider $\Uq=\Uq(\fg)$ where $\fg$ is either an orthogonal 
$g=\mathfrak{o}_{2n}$ and $g=\mathfrak{o}_{2n+1}$ or the symplectic
$g=\mathfrak{sp}_{2n}$ Lie algebra. Let $V=\Delta_q(\omega_1)$ be the quantized version of the 
corresponding vector representation. 
In both cases, $V$ is a $\Uq$-tilting module (for type $B$ and $q=1$ this requires 
$\Char(\K)\neq 2$, see \cite[Page 20]{ja1}) and hence, so is $T=V^{\otimes d}$. 
We first take $q=1$ 
and set $\delta=2n$ in case $\fg=\mathfrak{o}_{2n}$, and $\delta=2n+1$  
in case $\fg=\mathfrak{o}_{2n+1}$ and $\delta=-2n$ in 
case $\fg=\mathfrak{sp}_{2n}$ respectively. Then 
(see \cite[Theorem 1.4]{ddh} and \cite[Theorem 1.2]{dh} for infinite $\K$,
or \cite[Theorem 5.5]{es1} for $\K=\C$)
\begin{equation}\label{eq-schurweyl6}
\Phi_{\mathrm{Br}}\colon\cal{B}_d(\delta)
\twoheadrightarrow\End_{\Uu_1}(T)\quad
\text{and}\quad\Phi_{\mathrm{Br}}\colon\cal{B}_d(\delta)
\xrightarrow{\cong}\End_{\Uu_1}(T),\text{ if }n> d,
\end{equation}
where $\cal{B}_d(\delta)$ is the Brauer algebra in $d$ strands
(for $\fg\not=\mathfrak{o}_{2n}$ the isomorphism in \eqref{eq-schurweyl6} already 
holds for $n=d$). Thus, we get 
cellularity of $\cal{B}_d(\delta)$ 
by observing that in characteristic $p$ we can always assume that 
$n$ is large because $\cal{B}_d(\delta)=\cal{B}_d(\delta+p)$.

Similarly, let $\Uq=\Uq(\gll{n})$, $q\in\K^{\ast}$ arbitrary and
$T=\Delta_q(\omega_1)^{\otimes r}\otimes \Delta_q(\omega_{n-1})^{\otimes s}$. 
By \cite[Theorem 7.1 and Corollary 7.2]{dds} we have
\begin{equation}\label{eq-schurweyl7}
\Phi_{\mathrm{wBr}}\colon\cal{B}^n_{r,s}([n])
\twoheadrightarrow\End_{\Uq}(T)\quad\text{and}\quad\Phi_{\mathrm{wBr}}\colon\cal{B}^n_{r,s}([n])
\xrightarrow{\cong}\End_{\Uq}(T),\text{ if }n\geq r+s.
\end{equation}
Here $\cal{B}^n_{r,s}([n])$ is the quantized walled Brauer 
algebra for $[n]=q^{1-n}+\dots+q^{n-1}$. Since 
$T$ is a $\Uq$-tilting module, we get from \eqref{eq-schurweyl7}
cellularity of $\cal{B}^n_{r,s}([n])$ and of its quotients under $\Phi_{\mathrm{wBr}}$.

The walled Brauer algebra $\cal{B}^n_{r,s}(\delta)$ over 
$\K=\C$ for arbitrary parameter $\delta\in\Z$ 
appears as the centralizer of $\End_{\mathfrak{gl}(m|n)}(T)$ for 
$T=V^{\otimes r}\otimes (V^{*})^{\otimes s}$ where $V$ is the vector 
representation of the superalgebra $\mathfrak{gl}(m|n)$ with $\delta=m-n$. That is, we have
\begin{equation}\label{eq-schurweyl8}
\Phi_{\mathrm{s}}\colon\cal{B}^n_{r,s}(\delta)\twoheadrightarrow
\End_{\mathfrak{gl}(m|n)}(T)\;\;\text{and}\;\;
\Phi_{\mathrm{s}}\colon\cal{B}^n_{r,s}(\delta)\xrightarrow{\cong}
\End_{\mathfrak{gl}(m|n)}(T),\text{ if }(m+1)(n+1)\geq r+s,
\end{equation}
see \cite[Theorem 7.8]{bs5}. 
It can be shown that $T$ is a $\mathfrak{gl}(m|n)$-tilting module and thus, 
our main theorem applies and hence, by \eqref{eq-schurweyl8}, 
$\cal{B}^n_{r,s}(\delta)$ is cellular. 
Similarly for the quantized version.

Quantizing the Brauer case, taking $q\in\K^{\ast}$, 
$\fg$, $V=\Delta_q(\omega_1)$ and $T$ as before (without the 
restriction $\Char(\K)\neq 2$ for type $B$)
gives us a cellular 
structure on $\Endrr$. The algebra $\Endrr$ is a quotient of the 
Birman--Murakami--Wenzl algebra $\cal{BMW}_d(\delta)$ 
(for appropriate parameters), 
see \cite[(9.6)]{lz2} for the orthogonal case (which works 
for any $q\in\C-\{0,\pm 1\}$) 
and \cite[Theorem 1.5]{hu} for the symplectic case 
(which works 
for any $q\in\K^{\ast}-\{1\}$ and infinite $\K$). 
Again, taking 
$n\geq d$ (or $n> d$), we recover the cellularity of 
$\cal{BMW}_d(\delta)$.

\subsubsection{Infinite-dimensional modules --- highest weight categories}\label{subsub-cato}

Observe that our main theorem does not use the 
specific properties of $\Mod{\Uq}$, but works for any $\End_{\Mod{A}}(T)$ where 
$T$ is an $A$-tilting module for some finite-dimensional, quasi-hereditary algebra 
$A$ over $\K$ or $T\in\boldsymbol{\cal{C}}$ for some highest weight category 
$\boldsymbol{\cal{C}}$ in the sense 
of \cite{cps}. For the explicit construction of our basis we however need a 
notion like ``weight spaces'' such that 
\fullref{lem-cell1} makes sense.

The most famous example of such a category is the BGG category 
$\boldsymbol{\cal{O}}=\boldsymbol{\cal{O}}(\fg)$ attached to 
a complex semisimple or reductive Lie algebra $\fg$ with a 
corresponding Cartan $\mathfrak{h}$ and fixed Borel subalgebra $\fb$. We denote 
by $\Delta(\lambda)\in\boldsymbol{\cal{O}}$ the 
Verma module attached to $\lambda\in\fh$. 
In the same vein, pick a parabolic $\fp\supset\fb$ and denote for any $\fp$-dominant weight $\lambda$ the 
corresponding parabolic Verma module by $\Delta^{\fp}(\lambda)$. It is the unique quotient 
of the Verma module $\Delta(\lambda)$ which is locally $\fp$-finite, i.e. contained in 
the parabolic category 
$\boldsymbol{\cal{O}}^{\fp}=\boldsymbol{\cal{O}}^{\fp}(\fg)\subset \boldsymbol{\cal{O}}$ 
(see e.g. \cite{hum}).

There is a contravariant, character preserving duality functor 
${}^{\vee}\colon\boldsymbol{\cal{O}}^{\fp}\to\boldsymbol{\cal{O}}^{\fp}$ 
which allows us to set 
$\nabla^{\fp}(\lambda)=\Delta^{\fp}(\lambda)^{\vee}$.
Hence, we can play the same game again since the 
$\boldsymbol{\cal{O}}$-tilting theory works in a very similar 
fashion as for $\Mod{\Uq}$ (see \cite[Chapter 11]{hum} and the references therein). 
In particular, we have indecomposable $\boldsymbol{\cal{O}}$-tilting modules 
$T(\lambda)$ for any $\lambda\in\mathfrak{h}^{*}$. Similarly for 
$\boldsymbol{\cal{O}}^{\fp}$ 
giving an indecomposable $\boldsymbol{\cal{O}}^{\fp}$-tilting module 
$T(\lambda)$ for any $\fp$-dominant $\lambda\in\mathfrak{h}^{*}$.

We give a few examples where our approach leads to cellular 
structures on interesting algebras. For this purpose, 
let $\fp=\fb$ and $\lambda=0$. Then $T(0)$ has Verma 
factors of the form $\Delta(w.0)$ (for $w\in W$, where $W$ is the Weyl group associated to $\fg$). 
Each of these appears with multiplicity $1$. Hence, 
$\dim(\End_{\boldsymbol{\cal{O}}}(T(0)))=|W|$ by the analog 
of \eqref{eq-dimhom}. Then we have
$\End_{\boldsymbol{\cal{O}}}(T(0))\cong S(\fh)/S^W_+$. The algebra $S(\fh)/S^W_+$ is called 
the coinvariant algebra. (For the notation, the conventions and the 
result see \cite{soe1} --- this is Soergel's famous Endomorphismensatz.) Hence, our main 
theorem implies that $S(\fh)/S^W_+$ is cellular, 
which is no big surprise since all finite-dimensional, 
commutative algebras are cellular, see \cite[Proposition 3.5]{kx}.

There is also a quantum version of this result: replace 
$\boldsymbol{\cal{O}}$ by its quantum cousin 
$\boldsymbol{\cal{O}}_q$ from \cite{am} 
(which is the analog of $\boldsymbol{\cal{O}}$ for $\Uq(\fg)$). 
This works over any field $\K$ with $\Char(\K)=0$ 
and any $q\in\K^{\ast}-\{1\}$ (which can be deduced 
from Section 6 therein).
There is furthermore a characteristic $p$ version of this result: consider the
$G$-tilting module $T(p\rho)$ in the category of finite-dimensional $G$-modules 
(here $G$ is an almost simple, simply connected algebraic group over $\K$ 
with $\Char(\K)=p$). 
Its endomorphism algebra is isomorphic to 
the corresponding coinvariant algebra over $\K$, see \cite[Proposition 19.8]{ajs}.

Returning to $\K=\C$, we can generalize 
the example of the coinvariant algebra. To this end, note that, if $T$ is 
an $\boldsymbol{\cal{O}}^{\fp}$-tilting module, then so is $T\otimes M$ for 
any finite-dimensional $\fg$-module $M$, 
see \cite[Proposition 11.1 and Section 11.8]{hum} 
(and the references therein). Thus, 
$\End_{\boldsymbol{\cal{O}}^{\fp}}(T\otimes M)$ is cellular by our main 
theorem.

A special case is: $\fg$ is of classical type, 
$T=\Delta^{\fp}(\lambda)$ is simple (hence, $\boldsymbol{\cal{O}}^{\fp}$-tilting), 
$V$ is the vector 
representation of $\fg$ and $M=V^{\otimes d}$. Let first $\fg=\gll{n}$ 
with standard Borel $\fb$ 
and parabolic $\fp$ of block size $(n_1,\dots,n_{\ell})$. Then one 
can find a certain 
$\fp$-dominant weight $\lambda_{\mathrm{I}}$, called Irving-weight, such that 
$T=\Delta^{\fp}(\lambda_{\mathrm{I}})$ is $\boldsymbol{\cal{O}}^{\fp}$-tilting. Moreover, 
$\End_{\boldsymbol{\cal{O}}^{\fp}}(T\otimes V^{\otimes d})$ is isomorphic to a sum of blocks of 
cyclotomic quotients of the degenerate affine Hecke algebra 
$\cal{H}_d/\Pi_{i=1}^{\ell}(x_i-n_i)$, 
see \cite[Theorem 5.13]{bk}.
In the special case of level $\ell=2$, these algebras can be explicitly 
described in terms of generalizations of Khovanov's arc algebra 
(which Khovanov introduced in \cite{kh4} 
to give an algebraic structure underlying Khovanov homology 
and which categorifies the Temperley--Lieb algebra $\cal{TL}_d(\delta)$) 
and have an 
interesting representation theory, 
see \cite{bs1}, \cite{bs2}, \cite{bs3} and \cite{bs4}.
A consequence of this is that, using the results 
from \cite[Theorem 6.9]{sa} and \cite[Theorem 1.1]{sast}, one can realize the 
walled Brauer algebra from \fullref{subsub-brauer} for arbitrary parameter $\delta\in\Z$ 
as endomorphism algebras of some 
$\boldsymbol{\cal{O}}^{\fp}$-tilting module and hence, using our main theorem, 
deduce cellularity again.

If $\fg$ is of another classical type, then the role of the 
(cyclotomic quotients of the) degenerate affine Hecke algebra 
is played by (cyclotomic quotients of) 
degenerate BMW algebras or so-called (cyclotomic quotients of) 
$\VW$-algebras (also called Nazarov--Wenzl algebras). These are still poorly 
understood and technically quite involved, see \cite{amr}. In \cite{es} special examples 
of level $\ell=2$ quotients were studied and realized as endomorphism algebras 
of some $\boldsymbol{\cal{O}}^{\fp}(\soo{2n})$-tilting 
module 
$\Delta^{\fp}(\underline{\delta})\otimes V\in\boldsymbol{\cal{O}}^{\fp}(\soo{2n})$ 
where $V$ is the vector representation of $\soo{2n}$, 
$\underline{\delta}=\frac{\delta}{2}\sum_{i=1}^n\epsilon_i$ 
and $\fp$ is a maximal parabolic 
subalgebra of type $A$ 
(Theorem B loc. cit.). 
Hence, our theorem implies cellularity of these algebras.
Soergel's theorem is therefore just a shadow of a 
rich world of endomorphism algebras whose cellularity 
can be obtained from our approach.

Our methods also apply to (parabolic) category 
$\boldsymbol{\cal{O}}^{\fp}(\hat{\fg})$ attached 
to an affine Kac--Moody algebra $\hat{\fg}$ over $\K$ and 
related categories. In particular, one 
can consider a (level-dependent) quotient 
$\hat{\fg}_{\kappa}$ of $\Uu(\hat{\fg})$ and a category, denoted by
$\boldsymbol{\mathrm{O}}^{\nu,\kappa}_{\K,\tau}$, attached 
to it (we refer the reader to \cite[Sections 5.2 and 5.3]{rsvv} 
for the details). 
Then there 
is a subcategory 
$\boldsymbol{\mathrm{A}}^{\nu,\kappa}_{\K,\tau}\subset \boldsymbol{\mathrm{O}}^{\nu,\kappa}_{\K,\tau}$ 
and a $\boldsymbol{\mathrm{A}}^{\nu,\kappa}_{\K,\tau}$-tilting module 
$\boldsymbol{\mathrm{T}}_{\K,d}$
defined in Section 5.5 loc. cit. such that
\[
\Phi_{\mathrm{aff}}\colon\boldsymbol{\mathrm{H}}^s_{\K,d}
\rightarrow\End_{\boldsymbol{\mathrm{A}}^{\nu,\kappa}_{\K,\tau}}(\boldsymbol{\mathrm{T}}_{\K,d})\quad\text{and}\quad\Phi_{\mathrm{aff}}\colon\boldsymbol{\mathrm{H}}^s_{\K,d}
\xrightarrow{\cong}
\End_{\boldsymbol{\mathrm{A}}^{\nu,\kappa}_{\K,\tau}}(\boldsymbol{\mathrm{T}}_{\K,d}),\text{ if }\nu_p\geq d, p=1,\dots N,
\]
see \cite[Theorem 5.37 and Proposition 8.1]{rsvv}.
Here $\boldsymbol{\mathrm{H}}^s_{\K,d}$ denotes an appropriate cyclotomic quotient of the 
affine Hecke algebra. Again, our main theorem 
applies for $\boldsymbol{\mathrm{H}}^s_{\K,d}$ in case $\nu_p\geq d$.

\subsubsection{Graded cellular structures}\label{subsub-non}

A striking property which arises in the context of 
(parabolic) category $\boldsymbol{\cal{O}}$ (or $\boldsymbol{\cal{O}}^{\fp}$) is that 
all the endomorphism algebras from \fullref{subsub-cato} can be 
equipped with a $\Z$-grading as in \cite{st2} 
arising from the Koszul grading of category $\boldsymbol{\cal{O}}$ 
(or of $\boldsymbol{\cal{O}}^{\fp}$). 
We might choose 
our cellular basis compatible with this grading and obtain a grading on the endomorphism 
algebras turning them into graded cellular algebras in 
the sense of \cite[Definition 2.1]{hm}.

For the cyclotomic quotients this grading is non-trivial 
and in fact is the type $A$ KL--R grading 
in the spirit of Khovanov and Lauda and independently Rouquier 
(see \cite{kl1} and \cite{kl3} or \cite{rou}), which can be seen as 
a grading on cyclotomic quotients of degenerate affine Hecke algebras, see \cite{bk1}.
See \cite{bs3} for level $\ell=2$ and \cite{hm2} for all levels where the authors construct 
explicit graded cellular bases. For gradings on (cyclotomic quotients of)
$\VW$-algebras see 
\cite[Section 5]{es} and for gradings on Brauer algebras see \cite{ehst} or \cite{li}.

In the same spirit, it should be possible to obtain 
the higher level analogs of the generalizations of 
Khovanov's arc algebra, known as $\sll{n}$-web (or, 
alternatively, $\gll{n}$-web) algebras (see \cite{mpt} 
and \cite{mack1}), from our setup as well using the connections from cyclotomic KL--R algebras 
to these algebras in \cite{tub3} and \cite{tub4}. Although details still need to be worked out, 
this can be seen as the categorification of the connections to the spiders 
from \fullref{subsub-spid}: the spiders provide the setup to study the 
corresponding Reshetikhin--Turaev $\sll{n}$-link polynomials; the $\sll{n}$-web algebras 
provide the algebraic setup to study the 
Khovanov--Rozansky $\sll{n}$-link homologies. This would emphasize the connection 
between our work and low-dimensional topology.

\subsection{(Graded) cellular structures and the Temperley--Lieb algebras: a comparison}\label{sub-graded}
Finally we want to present one explicit 
example, the Temperley--Lieb algebras, which is 
of particular interest in low-dimensional 
topology and categorification. Our main goal 
is to construct new (graded) cellular bases, and use our 
approach to establish semisimplicity conditions, 
and construct and compute the dimensions 
of its simple modules in new ways.

We start by briefly recalling the necessary definitions. 
The reader unfamiliar with these algebras might consider
for example \cite[Section 6]{gl} (or \cite{astproofs}, where 
we recall the basics in detail using the usual Temperley--Lieb diagrams and our notation).

Fix $\delta=q+q^{-1}$ for $q\in\K^{\ast}$.\footnote{The $\sll{2}$ case works with any 
$q\in\K^{\ast}$, including even roots of unity, 
see e.g. \cite[Definition 2.3]{at}.}
Recall that the \textit{Temperley--Lieb algebra} $\cal{TL}_d(\delta)$ in $d$ strands 
with parameter $\delta$ is the free diagram algebra over $\K$ with basis 
consisting of all possible non-intersecting tangle diagrams with $d$ bottom and top 
boundary points modulo boundary 
preserving isotopy and the local relation for evaluating circles given 
by the parameter\footnote{We point out that there are two different 
conventions about circle evaluations in the literature: evaluating 
to $\delta$ or to $-\delta$. We use the first convention 
because we want to stay close to the cited literature.} $\delta$.

Recall from \fullref{subsub-othersl2} 
(whose notation we use now) that, by quantum Schur--Weyl duality, 
we can use \fullref{thm-cell} to obtain cellular bases of 
$\cal{TL}_d(\delta)\cong\End_{\Uq}(T)$ 
(we fix the isomorphism coming from quantum Schur--Weyl duality from now on). 
The aim now is to compare our 
cellular bases to the one given by 
Graham and Lehrer in \cite[Theorem 6.7]{gl}, where we point out 
that we do not obtain their cellular basis: 
our cellular basis depends for instance on 
whether $\cal{TL}_d(\delta)$ is semisimple or not. 
In the non-semisimple case, at least for $\K=\C$, we obtain a non-trivially $\Z$-graded 
cellular basis in the sense of \cite[Definition 2.1]{hm}, see 
\fullref{prop-cellTL4}.

Before stating our cellular basis, we provide a criterion which tells precisely whether 
$\cal{TL}_d(\delta)$ is semisimple or not. 
Recall that there is a known criteria for which Weyl modules $\Delta_q(i)$ 
are simple, see e.g. \cite[Proposition 2.7]{at}.

\begin{prop}(\textbf{Semisimplicity criterion for $\cal{TL}_d(\delta)$.})\label{prop-tlsemisimple}
We have the following.
\begin{enumerate}[label=(\alph*)]

\item \label{semia} Let $\delta\neq 0$. 
Then $\cal{TL}_d(\delta)$ is 
semisimple if and only if 
$[i]=q^{1-i}+\cdots+q^{i-1}\neq 0$ for all $i=1,\dots,d$ if and only if $q$ is 
not a root of unity with $d<l=\mathrm{ord}(q^2)$, or $q=1$ and $\Char(\K)>d$.

\item \label{semib} Let $\Char(\K)=0$. 
Then $\cal{TL}_d(0)$ is semisimple if and only if $d$ is odd (or $d=0$).

\item \label{semic} Let $\Char(\K)=p>0$. 
Then $\cal{TL}_d(0)$ is semisimple if and only if $d\in\{1,3,5,\dots,2p-1\}$ (or $d=0$).\makeqed
\end{enumerate}
\end{prop}

\begin{proof}
\ref{semia}: We want to show that $T=V^{\otimes d}$ decomposes into simple 
$\Uq$-modules if and only if $d<l$, or $q=1$ and $\Char(\K)>d$, which is clearly 
equivalent to the non-vanishing of the $[i]$'s.

Assume that $d<l$. Since the maximal $\Uq$-weight of 
$V^{\otimes d}$ is $d$ and since all Weyl $\Uq$-modules $\Delta_q(i)$ 
for $i<l$ are simple, 
we see that all indecomposable summands of 
$V^{\otimes d}$ are simple.

Otherwise, if $l\leq d$, then $T_q(d)$ (or $T_q(d-2)$ 
in the case $d\equiv -1\;\mathrm{mod}\;l$) is a non-simple, 
indecomposable summand of $V^{\otimes d}$ (note that this arguments 
fails if $l=2$, i.e. $\delta=0$).

The case $q=1$ works similarly, and 
we can now use \fullref{thm-cellsemisimple} to finish the proof of \ref{semia}.

\ref{semib}: Since $\delta=0$ if and only if $q=\pm\sqrt[2]{-1}$, we can use the 
linkage from e.g. \cite[Theorem 2.23]{at} in the case $l=2$ to 
see that $T=V^{\otimes d}$ decomposes 
into a direct sum of simple $\Uq$-modules if and only if $d$ is odd (or $d=0$).
This implies that $\cal{TL}_d(0)$ is semisimple if and only if $d$ is odd (or $d=0$) by 
\fullref{thm-cellsemisimple}.

\ref{semic}: If $\Char(\K)=p>0$ and $\delta=0$ (for $p=2$ this is equivalent to $q=1$), 
then we have $\Delta_q(i)\cong L_q(i)$ if and only if $i=0$ or $i\in\{2ap^{n}-1\mid n\in\Z_{\geq 0},1\leq a<p\}$. In 
particular, this means that for $d\geq 2$ we have that either $T_q(d)$ or $T_q(d-2)$ is 
a simple $\Uq$-module if and only if $d\in\{3,5,\dots,2p-1\}$. Hence, using the same reasoning 
as above, we see that $T=V^{\otimes d}$ is 
semisimple if and only if $d\in\{1,3,5,\dots,2p-1\}$ (or $d=0$). 
By \fullref{thm-cellsemisimple} we see that $\cal{TL}_d(0)$ is semisimple 
if and only if $d\in\{1,3,5,\dots,2p-1\}$ (or $d=0$).
\end{proof}

\begin{ex}\label{ex-cpoly}
We have that $[k]\neq 0$ for all $k=1,2,3$ is satisfied 
if and only if $q$ is not a fourth or a sixth root of unity.
By \fullref{prop-tlsemisimple} we see that $\cal{TL}_3(\delta)$ is semisimple as long as 
$q$ is not one of these values from above. The other way around is only true for $q$ being 
a sixth root of unity (the conclusion from semisimplicity to non-vanishing of the quantum numbers 
above does not work in the case $q=\pm\sqrt[2]{-1}$).
\end{ex}

\begin{rem}\label{rem-deltazero}
The semisimplicity criterion for $\cal{TL}_d(\delta)$ 
was already already found, using quite different methods, 
in \cite[Section 5]{west} in the case $\delta\neq 0$, and in the case $\delta=0$
in \cite[Chapter 7]{martin} 
or \cite[above Proposition 4.9]{rsa}.
For us it is an easy application of \fullref{thm-cellsemisimple}.
\end{rem}

A direct consequence of \fullref{prop-tlsemisimple} is that the 
Temperley--Lieb algebra $\cal{TL}_d(\delta)$ for $q\in\K^{\ast}-\{1\}$ not 
a root of unity is semisimple (or $q=\pm 1$ and $\Char(\K)=0$), 
regardless of $d$.

\subsubsection{Temperley--Lieb algebra: the semisimple case}\label{subsub-semiTL}

Assume $q\in\K^{\ast}-\{1\}$ is not a root of unity 
(or $q=\pm 1$ and $\Char(\K)=0$). 
Thus, we are in the semisimple case.

Let us compare our cell datum $(\P,\Tt,\Ct,\I)$ to
the one of Graham and Lehrer 
(indicated by a subscript $\TL$) 
from \cite[Section 6]{gl}. They have the poset $\P_{\TL}$ consisting of 
all length-two partitions of $d$, and we have the poset $\P$ 
consisting of all $\lambda\in X^+$ such that $\Dl$ is a factor of $T$.
The two sets 
are clearly the same: an element $\lambda=(\lambda_1,\lambda_2)\in\P_{\TL}$ 
corresponds to $\lambda_1-\lambda_2\in\P$. 
Similarly, an inductive reasoning shows that 
$\Tt_{\TL}$
(standard fillings of the Young 
diagram associated to $\lambda$) is also the same as our $\Tt$ 
(to see this one can use the facts listed in \cite[Section 2]{at}).
One directly checks that the $\K$-linear anti-involution 
$\I_{\TL}$ (turning diagrams upside-down) is also our involution $\I$.
Thus, except for $\Ct$ and $\Ct_{\TL}$, the cell 
data agree.

In order to state how our cellular bases 
for $\cal{TL}_d(\delta)$ look like, 
recall that the so-called 
\textit{generalized Jones--Wenzl projectors} $JW_{\vec{\epsilon}}$
are indexed by $d$-tuples (with $d>0$) of the form 
$\vec{\epsilon}=(\epsilon_1,\dots, \epsilon_d)\in\{\pm 1\}^d$ such that 
$\sum_{j=1}^k\epsilon_j\geq 0$ for all $k=1,\dots,d$, 
see e.g. \cite[Section 2]{ch}.
In case $\vec{\epsilon}=(1,\dots,1)$, one recovers the 
usual Jones--Wenzl projectors introduced by Jones in \cite{jones} 
and then further studied by Wenzl in \cite{wenzl}.

Now, in \cite[Proposition 2.19 and Theorem 2.20]{ch} it is 
shown that there exist non-zero scalars $a_{\vec{\epsilon}}\in\K$ such that 
$JW_{\vec{\epsilon}}^{\prime}=a_{\vec{\epsilon}}JW_{\vec{\epsilon}}$ are well-defined 
idempotents forming a complete set of mutually orthogonal, primitive idempotents in 
$\cal{TL}_d(\delta)$. (The authors of \cite{ch} work 
over $\C$, but as long as $q\in\K^{\ast}-\{1\}$ is not a root 
of unity their arguments work in our setup as well.) These project 
to the summands of $T=V^{\otimes d}$ of the 
form $\Delta_q(i)$ for $i=\sum_{j=1}^k\epsilon_j$. In particular, the usual 
Jones--Wenzl projectors project to the highest weight summand 
$\Delta_q(d)$ of $T=V^{\otimes d}$.

\begin{prop}(\textbf{(New) cellular bases.})\label{prop-cellTL2}
The datum given by the quadruple 
$(\P,\Tt,\Ct,\I)$ 
for $\cal{TL}_d(\delta)\cong\End_{\Uq}(T)$ is 
a cell datum for $\cal{TL}_d(\delta)$.
Moreover, $\Ct\neq \Ct_{\TL}$ for all $d>1$ and all choices involved in the definition 
of $\mathrm{im}(\Ct)$.
In particular, there is a choice such that all generalized Jones--Wenzl projectors 
$JW_{\vec{\epsilon}}^{\prime}$ are 
part of $\mathrm{im}(\Ct)$.\makeqed
\end{prop}

\begin{proof}
That we get a cell datum as stated follows from \fullref{thm-cell} and 
the discussion above.

That our cellular basis $\Ct$ will never be $\Ct_{\TL}$ for $d>1$ is due to the fact that 
Graham and Lehrer's cellular basis always contains the identity 
(which corresponds to the unique standard filling of the Young diagram associated to $\lambda=(d,0)$).

In contrast, let $\lambda_k=(d-k,k)$ for $0\leq k\leq\lfloor \frac{d}{2}\rfloor$. Then
\begin{equation}\label{eq-something2}
T=V^{\otimes d}
\cong \Delta_q(d)\oplus\bigoplus_{0< k\leq\lfloor \frac{d}{2}\rfloor}\Delta_q(d-2k)^{\oplus m_{\lambda_k}}
\end{equation}
for some multiplicities $m_{\lambda_k}\in\Z_{>0}$, we see that 
for $d>1$ the identity 
is never part of any of our bases: all the $\Delta_q(i)$'s 
are simple $\Uq$-modules and each 
$c_{ij}^{k}$ factors only through $\Delta_q(k)$.
In particular, the basis element $c_{11}^{\lambda}$ 
for $\lambda=\lambda_d$ has to be (a scalar multiple) of $JW_{(1,\dots,1)}$.

As in \fullref{subsub-semi} we can choose for $\Ct$ an Artin--Wedderburn basis 
of $\cal{TL}_d(\delta)\cong\Endrr$. 
Hence, by the above, the corresponding basis consists of 
the projectors $JW_{\vec{\epsilon}}$. 
\end{proof}

Note the following classification result 
(see for example \cite[Corollary 5.2]{rsa} for $\K=\C$).

\begin{cor}\label{cor-dim1}
We have a complete set of pairwise non-isomorphic, simple $\cal{TL}_d(\delta)$-modules 
$L(\lambda)$, where $\lambda=(\lambda_1,\lambda_2)$ 
is a length-two partition of $d$.
Moreover,
$
\dim(L(\lambda))=|\mathrm{Std}(\lambda)|,
$
where $\mathrm{Std}(\lambda)$ is the set of 
all standard tableaux of shape $\lambda$.\makeqed
\end{cor}

\begin{proof}
This follows directly from \fullref{prop-cellTL2} and 
\fullref{thm-cellclass} and \fullref{thm-dim} because 
we have
$m_{\lambda}=|\mathrm{Std}(\lambda)|$.
\end{proof}

\subsubsection{Temperley--Lieb algebra: the non-semisimple case}\label{subsub-nonsemiTL}

Let us assume that we have fixed $q\in\K^{\ast}-\{1,\pm\sqrt[2]{-1}\}$ to be a critical 
value such that 
$[k]=0$ for some $k=1,\dots,d$. Then, by \fullref{prop-tlsemisimple}, the 
algebra $\cal{TL}_d(\delta)$ is no longer semisimple.
In particular, to the best of our knowledge, there is no 
diagrammatic analog of the Jones--Wenzl projectors in general.

\begin{prop}(\textbf{(New) cellular basis --- the second.})\label{prop-cellTL3}
The datum $(\P,\Tt,\Ct,\I)$ with $\Ct$ as in \fullref{thm-cell} 
for $\cal{TL}_d(\delta)\cong\End_{\Uq}(T)$ is a cell datum for $\cal{TL}_d(\delta)$.
Moreover, $\Ct\neq \Ct_{\TL}$ for all $d>1$ and all choices involved in 
the definition of our basis.
Thus, there is a choice such that all generalized, non-semisimple 
Jones--Wenzl projectors are 
part of $\mathrm{im}(\Ct)$.\makeqed
\end{prop}

\begin{proof}
As in the proof of \fullref{prop-cellTL2} and left to the reader.
\end{proof}

Hence, 
directly from \fullref{prop-cellTL3} and 
\fullref{thm-cellclass} and \fullref{thm-dim}, we obtain:

\begin{cor}\label{cor-dim2} 
We have a complete set of pairwise non-isomorphic, simple $\cal{TL}_d(\delta)$-modules 
$L(\lambda)$, 
where $\lambda=(\lambda_1,\lambda_2)$ is a length-two partition of $d$.
Moreover, 
$
\dim(L(\lambda))=m_{\lambda},
$
where $m_{\lambda}$ is the multiplicity of 
$T_q(\lambda_1-\lambda_2)$ 
as a summand of $T=V^{\otimes d}$.\qedmake
\end{cor}

Note that we can do better:  
one gets a decompositions
\begin{equation}\label{eq-gradeddecom}
\T\cong \T_{-1}\oplus \T_{0} \oplus \T_{1} \oplus\cdots\oplus \T_{l-3} \oplus \T_{l-2} \oplus \T_{l-1},
\end{equation}
where the blocks $\T_{-1}$ and $\T_{l-1}$ are semisimple if $\K=\C$. 
(This follows from 
the linkage principle. For notation and the statement see \cite[Section 2]{at}.)

Fix $\K=\C$. As explained in \cite[Section 3.5]{at} each block in the 
decomposition \eqref{eq-gradeddecom} can be equipped with a non-trivial $\Z$-grading coming 
from the zig-zag algebra from \cite{hk1}. Hence, we have the following.

\begin{lem}\label{lem-grading}
The $\C$-algebra $\Endrr$ can be equipped with a non-trivial $\Z$-grading. 
Thus, $\cal{TL}_d(\delta)$ over $\C$ 
can be equipped with a non-trivial $\Z$-grading.\makeqed
\end{lem}

\begin{proof}
The second statement follows directly from the first using quantum Schur--Weyl duality. 
Hence, we only need to show the first.

Note that $T=V^{\otimes d}$ decomposes as in \eqref{eq-something2}, 
but with $T_q(k)$'s instead of $\Delta_q(k)$'s, and we can order 
this decomposition by blocks. Each block carries a $\Z$-grading 
coming from the zig-zag 
algebra, as explained in \cite[Section 3]{at}.
In particular, we can choose the basis elements $c_{ij}^{\lambda}$ in such a 
way that we get the $\Z$-graded basis obtained in Corollary 4.23 therein.
Since there is no interaction between different blocks, the statement follows.
\end{proof}

Recall from \cite[Definition 2.1]{hm} that a $\Z$-graded cell datum of a 
$\Z$-graded algebra is a cell datum for the algebra together with 
an additional \textit{degree function}
$
\mathrm{deg}\colon\coprod_{\lambda\in\P}\Ttt\to \Z,
$
such that $\mathrm{deg}(c_{ij}^{\lambda})=\mathrm{deg}(i)+\mathrm{deg}(j)$.
For us the choice of $\mathrm{deg}(\cdot)$ is as follows.

If $\lambda\in\P$ is in one of the semisimple 
blocks, then we simply set $\mathrm{deg}(i)=0$ for all $i\in\Ttt$.

Assume that $\lambda\in\P$ is not in the semisimple blocks.
It is known that 
every $\Tl$ has precisely two Weyl factors. The $g^{\lambda}_i$ that map 
$\Dl$ into a higher $T_q(\mu)$ should be indexed 
by a $1$-colored $i$ whereas 
the $g^{\lambda}_i$ mapping $\Dl$ into $\Tl$ should have 
$0$-colored $i$. Similarly for the $f^{\lambda}_j$'s.
Then the degree of the elements $i\in\Ttt$ 
should be the corresponding color.
We get the following. (Here $\Ct$ is as in \fullref{thm-cell}.)

\begin{prop}(\textbf{Graded cellular basis.})\label{prop-cellTL4}
The datum $(\P,\Tt,\Ct,\I)$ 
supplemented with the function $\mathrm{deg}(\cdot)$ from above is a $\Z$-graded cell 
datum for the $\C$-algebra $\cal{TL}_d(\delta)\cong\Endrr$.\makeqed
\end{prop}

\begin{proof}
The hardest part is cellularity which directly follows 
from \fullref{thm-cell}. That the quintuple 
$(\P,\Tt,\Ct,\I,\mathrm{deg})$
gives a $\Z$-graded cell 
datum follows from the construction.
\end{proof}

\begin{rem}
Our grading and the one 
found by Plaza and Ryom-Hansen in \cite{prh} agree (up to a shift of 
the indecomposable summands). To see this, note that 
our algebra is isomorphic to the algebra $K_{1,n}$ studied in \cite{bs1} which 
is by (4.8) therein and \cite[Theorem 6.3]{bs3} a quotient of some 
particular cyclotomic 
KL--R algebra (the compatibility of the grading 
follows for example from \cite[Corollary B.6]{hm2}). The same holds, by 
construction, for the grading in \cite{prh}.
\end{rem}
%
\bibliographystyle{plainurl}
\bibliography{cell-tilt}

\begin{thebibliography}{10}

\bibitem{an1}
H.H. Andersen.
\newblock Tensor products of quantized tilting modules.
\newblock {\em Comm. Math. Phys.}, 149(1):149--159, 1992.

\bibitem{an0}
H.H. Andersen$\phantom{a}\!\!\!$.
\newblock Filtrations and tilting modules.
\newblock {\em Ann. Sci. \'Ecole Norm. Sup. (4)}, 30(3):353--366, 1997.
\newblock \href {http://dx.doi.org/10.1016/S0012-9593(97)89924-7}
  {\path{doi:10.1016/S0012-9593(97)89924-7}}.

\bibitem{an2}
H.H. Andersen$\phantom{a}\!\!\!$.
\newblock The strong linkage principle for quantum groups at roots of $1$.
\newblock {\em J. Algebra}, 260(1):2--15, 2003.
\newblock \href {http://dx.doi.org/10.1016/S0021-8693(02)00618-X}
  {\path{doi:10.1016/S0021-8693(02)00618-X}}.

\bibitem{ajs}
H.H. Andersen$\phantom{a}\!\!\!$, J.C. Jantzen, and W.~Soergel.
\newblock Representations of quantum groups at a {$p$}th root of unity and of
  semisimple groups in characteristic {$p$}: independence of {$p$}.
\newblock {\em Ast\'erisque}, (220):321, 1994.

\bibitem{alz}
H.H. Andersen$\phantom{a}\!\!\!$, G.~Lehrer, and R.~Zhang.
\newblock Cellularity of certain quantum endomorphism algebras.
\newblock {\em Pacific J. Math.}, 279(1-2):11--35, 2015.
\newblock URL: \url{http://arxiv.org/abs/1303.0984}, \href
  {http://dx.doi.org/10.2140/pjm.2015.279.11}
  {\path{doi:10.2140/pjm.2015.279.11}}.

\bibitem{am}
H.H. Andersen$\phantom{a}\!\!\!$ and V.~Mazorchuk.
\newblock Category {$\mathcal{O}$} for quantum groups.
\newblock {\em J. Eur. Math. Soc. (JEMS)}, 17(2):405--431, 2015.
\newblock URL: \url{http://arxiv.org/abs/1105.5500}, \href
  {http://dx.doi.org/10.4171/JEMS/506} {\path{doi:10.4171/JEMS/506}}.

\bibitem{apw}
H.H. Andersen$\phantom{a}\!\!\!$, P.~Polo, and K.X. Wen.
\newblock Representations of quantum algebras.
\newblock {\em Invent. Math.}, 104(1):1--59, 1991.
\newblock \href {http://dx.doi.org/10.1007/BF01245066}
  {\path{doi:10.1007/BF01245066}}.

\bibitem{astproofs}
H.H. Andersen$\phantom{a}\!\!\!$, C.~Stroppel, and D.~Tubbenhauer.
\newblock Additional notes for the paper ``{C}ellular structures using
  $\textbf{U}_q$-tilting modules''.
\newblock 2015.
\newblock URL:
  \url{http://pure.au.dk/portal/files/100562565/cell_tilt_proofs_1.pdf},
  \url{http://www.math.uni-bonn.de/ag/stroppel/cell-tilt-proofs_neu.pdf},
  \url{http://www.math.uni-bonn.de/people/dtubben/cell-tilt-proofs.pdf}.

\bibitem{ast}
H.H. Andersen$\phantom{a}\!\!\!$, C.~Stroppel, and D.~Tubbenhauer.
\newblock Semisimplicity of {H}ecke and (walled) {B}rauer algebras.
\newblock {\em J. Aust. Math. Soc.}, 103(1):1--44, 2017.
\newblock URL: \url{http://arxiv.org/abs/1507.07676}, \href
  {http://dx.doi.org/10.1017/S1446788716000392}
  {\path{doi:10.1017/S1446788716000392}}.

\bibitem{at}
H.H. Andersen$\phantom{a}\!\!\!$ and D.~Tubbenhauer.
\newblock Diagram categories for $\textbf{U}_q$-tilting modules at roots of
  unity.
\newblock {\em Transform. Groups}, 22(1):29--89, 2017.
\newblock URL: \url{https://arxiv.org/abs/1409.2799}, \href
  {http://dx.doi.org/10.1007/s00031-016-9363-z}
  {\path{doi:10.1007/s00031-016-9363-z}}.

\bibitem{ari}
S.~Ariki.
\newblock Cyclotomic $q$-{S}chur algebras as quotients of quantum algebras.
\newblock {\em J. Reine Angew. Math.}, 513:53--69, 1999.
\newblock \href {http://dx.doi.org/10.1515/crll.1999.063}
  {\path{doi:10.1515/crll.1999.063}}.

\bibitem{ak}
S.~Ariki and K.~Koike.
\newblock A {H}ecke algebra of $(\mathbb{Z}/r\mathbb{Z})\wr \mathfrak{S}_n$ and
  construction of its irreducible representations.
\newblock {\em Adv. Math.}, 106(2):216--243, 1994.
\newblock \href {http://dx.doi.org/10.1006/aima.1994.1057}
  {\path{doi:10.1006/aima.1994.1057}}.

\bibitem{amr}
S.~Ariki, A.~Mathas, and H.~Rui.
\newblock Cyclotomic {N}azarov--{W}enzl algebras.
\newblock {\em Nagoya Math. J.}, 182:47--134, 2006.
\newblock URL: \url{http://arxiv.org/abs/math/0506467}.

\bibitem{bw}
J.S. Birman and H.~Wenzl.
\newblock Braids, link polynomials and a new algebra.
\newblock {\em Trans. Amer. Math. Soc.}, 313(1):249--273, 1989.
\newblock \href {http://dx.doi.org/10.2307/2001074}
  {\path{doi:10.2307/2001074}}.

\bibitem{bra}
R.~Brauer.
\newblock On algebras which are connected with the semisimple continuous
  groups.
\newblock {\em Ann. of Math. (2)}, 38(4):857--872, 1937.
\newblock \href {http://dx.doi.org/10.2307/1968843}
  {\path{doi:10.2307/1968843}}.

\bibitem{bk1}
J.~Brundan and A.~Kleshchev.
\newblock Blocks of cyclotomic {H}ecke algebras and {K}hovanov--{L}auda
  algebras.
\newblock {\em Invent. Math.}, 178(3):451--484, 2009.
\newblock URL: \url{http://arxiv.org/abs/0808.2032}, \href
  {http://dx.doi.org/10.1007/s00222-009-0204-8}
  {\path{doi:10.1007/s00222-009-0204-8}}.

\bibitem{bk}
J.~Brundan$\phantom{a}\!\!\!$ and A.~Kleshchev.
\newblock Schur--{W}eyl duality for higher levels.
\newblock {\em Selecta Math. (N.S.)}, 14(1):1--57, 2008.
\newblock URL: \url{http://arxiv.org/abs/math/0605217}, \href
  {http://dx.doi.org/10.1007/s00029-008-0059-7}
  {\path{doi:10.1007/s00029-008-0059-7}}.

\bibitem{bs5}
J.~Brundan$\phantom{a}\!\!\!$ and C.~Stroppel.
\newblock Gradings on walled {B}rauer algebras and {K}hovanov's arc algebra.
\newblock {\em Adv. Math.}, 231(2):709--773, 2012.
\newblock URL: \url{http://arxiv.org/abs/1107.0999}, \href
  {http://dx.doi.org/10.1016/j.aim.2012.05.016}
  {\path{doi:10.1016/j.aim.2012.05.016}}.

\bibitem{bs1}
J.~Brundan$\phantom{a}\!\!\!$$\phantom{a}\!\!\!$ and C.~Stroppel.
\newblock Highest weight categories arising from {K}hovanov's diagram algebra
  {I}: cellularity.
\newblock {\em Mosc. Math. J.}, 11(4):685--722, 821--822, 2011.
\newblock URL: \url{http://arxiv.org/abs/0806.1532}.

\bibitem{bs2}
J.~Brundan$\phantom{a}\!\!\!$$\phantom{a}\!\!\!$ and
  C.~Stroppel$\phantom{a}\!\!\!$.
\newblock Highest weight categories arising from {K}hovanov's diagram algebra
  {II}: {K}oszulity.
\newblock {\em Transform. Groups}, 15(1):1--45, 2010.
\newblock URL: \url{http://arxiv.org/abs/0806.3472}, \href
  {http://dx.doi.org/10.1007/s00031-010-9079-4}
  {\path{doi:10.1007/s00031-010-9079-4}}.

\bibitem{bs3}
J.~Brundan$\phantom{a}\!\!\!$$\phantom{a}\!\!\!$ and
  C.~Stroppel$\phantom{a}\!\!\!$.
\newblock Highest weight categories arising from {K}hovanov's diagram algebra
  {III}: category $\mathcal{O}$.
\newblock {\em Represent. Theory}, 15:170--243, 2011.
\newblock URL: \url{http://arxiv.org/abs/0812.1090}, \href
  {http://dx.doi.org/10.1090/S1088-4165-2011-00389-7}
  {\path{doi:10.1090/S1088-4165-2011-00389-7}}.

\bibitem{bs4}
J.~Brundan$\phantom{a}\!\!\!$$\phantom{a}\!\!\!$ and
  C.~Stroppel$\phantom{a}\!\!\!$.
\newblock Highest weight categories arising from {K}hovanov's diagram algebra
  {IV}: the general linear supergroup.
\newblock {\em J. Eur. Math. Soc. (JEMS)}, 14(2):373--419, 2012.
\newblock URL: \url{http://arxiv.org/abs/0907.2543}, \href
  {http://dx.doi.org/10.4171/JEMS/306} {\path{doi:10.4171/JEMS/306}}.

\bibitem{ckm}
S.~Cautis, J.~Kamnitzer, and S.~Morrison.
\newblock Webs and quantum skew {H}owe duality.
\newblock {\em Math. Ann.}, 360(1-2):351--390, 2014.
\newblock URL: \url{http://arxiv.org/abs/1210.6437}, \href
  {http://dx.doi.org/10.1007/s00208-013-0984-4}
  {\path{doi:10.1007/s00208-013-0984-4}}.

\bibitem{cps}
E.~Cline, B.~Parshall, and L.~Scott.
\newblock Finite-dimensional algebras and highest weight categories.
\newblock {\em J. Reine Angew. Math.}, 391:85--99, 1988.

\bibitem{ch}
B.~Cooper and M.~Hogancamp.
\newblock An exceptional collection for {K}hovanov homology.
\newblock {\em Algebr. Geom. Topol.}, 15(5):2659--2707, 2015.
\newblock URL: \url{http://arxiv.org/abs/1209.1002}, \href
  {http://dx.doi.org/10.2140/agt.2015.15.2659}
  {\path{doi:10.2140/agt.2015.15.2659}}.

\bibitem{ddh}
R.~Dipper, S.~Doty, and J.~Hu.
\newblock Brauer algebras, symplectic {S}chur algebras and {S}chur--{W}eyl
  duality.
\newblock {\em Trans. Amer. Math. Soc.}, 360(1):189--213 (electronic), 2008.
\newblock URL: \url{http://arxiv.org/abs/math/0503545}, \href
  {http://dx.doi.org/10.1090/S0002-9947-07-04179-7}
  {\path{doi:10.1090/S0002-9947-07-04179-7}}.

\bibitem{dds}
R.~Dipper, S.~Doty, and F.~Stoll.
\newblock The quantized walled {B}rauer algebra and mixed tensor space.
\newblock {\em Algebr. Represent. Theory}, 17(2):675--701, 2014.
\newblock URL: \url{http://arxiv.org/abs/0806.0264}, \href
  {http://dx.doi.org/10.1007/s10468-013-9414-2}
  {\path{doi:10.1007/s10468-013-9414-2}}.

\bibitem{djm}
R.~Dipper, G.~James, and A.~Mathas.
\newblock Cyclotomic $q$-{S}chur algebras.
\newblock {\em Math. Z.}, 229(3):385--416, 1998.
\newblock \href {http://dx.doi.org/10.1007/PL00004665}
  {\path{doi:10.1007/PL00004665}}.

\bibitem{don}
S.~Donkin.
\newblock On tilting modules for algebraic groups.
\newblock {\em Math. Z.}, 212(1):39--60, 1993.
\newblock \href {http://dx.doi.org/10.1007/BF02571640}
  {\path{doi:10.1007/BF02571640}}.

\bibitem{don1}
S.~Donkin.
\newblock {\em The $q$-{S}chur algebra}, volume 253 of {\em London Mathematical
  Society Lecture Note Series}.
\newblock Cambridge University Press, Cambridge, 1998.
\newblock \href {http://dx.doi.org/10.1017/CBO9780511600708}
  {\path{doi:10.1017/CBO9780511600708}}.

\bibitem{dh}
S.~Doty and J.~Hu.
\newblock Schur--{W}eyl duality for orthogonal groups.
\newblock {\em Proc. Lond. Math. Soc. (3)}, 98(3):679--713, 2009.
\newblock URL: \url{http://arxiv.org/abs/0712.0944}, \href
  {http://dx.doi.org/10.1112/plms/pdn044} {\path{doi:10.1112/plms/pdn044}}.

\bibitem{dps}
J.~Du, B.~Parshall, and L.~Scott.
\newblock Quantum {W}eyl reciprocity and tilting modules.
\newblock {\em Comm. Math. Phys.}, 195(2):321--352, 1998.
\newblock \href {http://dx.doi.org/10.1007/s002200050392}
  {\path{doi:10.1007/s002200050392}}.

\bibitem{es}
M.~Ehrig and C.~Stroppel.
\newblock Nazarov--{W}enzl algebras, coideal subalgebras and categorified skew
  {H}owe duality.
\newblock 2013.
\newblock URL: \url{http://arxiv.org/abs/1310.1972}.

\bibitem{ehst}
M.~Ehrig and C.~Stroppel.
\newblock Koszul {G}radings on {B}rauer {A}lgebras.
\newblock {\em Int. Math. Res. Not. IMRN}, (13):3970--4011, 2016.
\newblock URL: \url{http://arxiv.org/abs/1504.03924}, \href
  {http://dx.doi.org/10.1093/imrn/rnv267} {\path{doi:10.1093/imrn/rnv267}}.

\bibitem{es1}
M.~Ehrig and C.~Stroppel.
\newblock Schur--{W}eyl duality for the {B}rauer algebra and the
  ortho-symplectic {L}ie superalgebra.
\newblock {\em Math. Z.}, 284(1-2):595--613, 2016.
\newblock URL: \url{http://arxiv.org/abs/1412.7853}, \href
  {http://dx.doi.org/10.1007/s00209-016-1669-y}
  {\path{doi:10.1007/s00209-016-1669-y}}.

\bibitem{elias}
B.~Elias.
\newblock Light ladders and clasp conjectures.
\newblock 2015.
\newblock URL: \url{http://arxiv.org/abs/1510.06840}.

\bibitem{erd}
K.~Erdmann.
\newblock Tensor products and dimensions of simple modules for symmetric
  groups.
\newblock {\em Manuscripta Math.}, 88(3):357--386, 1995.
\newblock \href {http://dx.doi.org/10.1007/BF02567828}
  {\path{doi:10.1007/BF02567828}}.

\bibitem{gl}
J.J. Graham and G.~Lehrer.
\newblock Cellular algebras.
\newblock {\em Invent. Math.}, 123(1):1--34, 1996.
\newblock \href {http://dx.doi.org/10.1007/BF01232365}
  {\path{doi:10.1007/BF01232365}}.

\bibitem{hr}
T.~Halverson and A.~Ram.
\newblock $q$-rook monoid algebras, {H}ecke algebras, and {S}chur--{W}eyl
  duality.
\newblock {\em Zap. Nauchn. Sem. S.-Peterburg. Otdel. Mat. Inst. Steklov.
  (POMI)}, 283(Teor. Predst. Din. Sist. Komb. i Algoritm. Metody. 6):224--250,
  262--263, 2001.
\newblock URL: \url{http://arxiv.org/abs/math/0401330}, \href
  {http://dx.doi.org/10.1023/B:JOTH.0000024623.99412.13}
  {\path{doi:10.1023/B:JOTH.0000024623.99412.13}}.

\bibitem{hu}
J.~Hu.
\newblock B{MW} algebra, quantized coordinate algebra and type {$C$}
  {S}chur--{W}eyl duality.
\newblock {\em Represent. Theory}, 15:1--62, 2011.
\newblock URL: \url{http://arxiv.org/abs/0708.3009}, \href
  {http://dx.doi.org/10.1090/S1088-4165-2011-00369-1}
  {\path{doi:10.1090/S1088-4165-2011-00369-1}}.

\bibitem{hm}
J.~Hu and A.~Mathas.
\newblock Graded cellular bases for the cyclotomic
  {K}hovanov--{L}auda--{R}ouquier algebras of type ${A}$.
\newblock {\em Adv. Math.}, 225(2):598--642, 2010.
\newblock URL: \url{http://arxiv.org/abs/0907.2985}, \href
  {http://dx.doi.org/10.1016/j.aim.2010.03.002}
  {\path{doi:10.1016/j.aim.2010.03.002}}.

\bibitem{hm2}
J.~Hu and A.~Mathas.
\newblock Quiver {S}chur algebras for linear quivers.
\newblock {\em Proc. Lond. Math. Soc. (3)}, 110(6):1315--1386, 2015.
\newblock URL: \url{http://arxiv.org/abs/1110.1699}, \href
  {http://dx.doi.org/10.1112/plms/pdv007} {\path{doi:10.1112/plms/pdv007}}.

\bibitem{hs}
J.~Hu and F.~Stoll.
\newblock On double centralizer properties between quantum groups and
  {A}riki--{K}oike algebras.
\newblock {\em J. Algebra}, 275(1):397--418, 2004.
\newblock \href {http://dx.doi.org/10.1016/j.jalgebra.2003.10.026}
  {\path{doi:10.1016/j.jalgebra.2003.10.026}}.

\bibitem{hk1}
R.S. Huerfano and M.~Khovanov.
\newblock A category for the adjoint representation.
\newblock {\em J. Algebra}, 246(2):514--542, 2001.
\newblock URL: \url{https://arxiv.org/abs/math/0002060}, \href
  {http://dx.doi.org/10.1006/jabr.2001.8962}
  {\path{doi:10.1006/jabr.2001.8962}}.

\bibitem{hum}
J.E. Humphreys.
\newblock {\em Representations of semisimple {L}ie algebras in the {BGG}
  category {$\mathcal{O}$}}, volume~94 of {\em Graduate Studies in
  Mathematics}.
\newblock American Mathematical Society, Providence, RI, 2008.
\newblock \href {http://dx.doi.org/10.1090/gsm/094}
  {\path{doi:10.1090/gsm/094}}.

\bibitem{ja1}
J.C. Jantzen.
\newblock Darstellungen halbeinfacher algebraischer {G}ruppen und zugeordnete
  kontravariante {F}ormen.
\newblock {\em Bonn. Math. Schr.}, (67):v+124, 1973.

\bibitem{ja}
J.C. Jantzen.
\newblock {\em Lectures on quantum groups}, volume~6 of {\em Graduate Studies
  in Mathematics}.
\newblock American Mathematical Society, Providence, RI, 1996.

\bibitem{jarag}
J.C. Jantzen.
\newblock {\em Representations of algebraic groups}, volume 107 of {\em
  Mathematical Surveys and Monographs}.
\newblock American Mathematical Society, Providence, RI, second edition, 2003.

\bibitem{jones}
V.F.R. Jones.
\newblock Index for subfactors.
\newblock {\em Invent. Math.}, 72(1):1--25, 1983.
\newblock \href {http://dx.doi.org/10.1007/BF01389127}
  {\path{doi:10.1007/BF01389127}}.

\bibitem{kalu}
D.~Kazhdan and G.~Lusztig.
\newblock Tensor structures arising from affine {L}ie algebras. {I}, {II},
  {III}, {IV}.
\newblock {\em J. Amer. Math. Soc.}, 6,7(4,2):905--947, 949--1011, 335--381,
  383--453, 1993,1994.

\bibitem{kh4}
M.~Khovanov.
\newblock A functor-valued invariant of tangles.
\newblock {\em Algebr. Geom. Topol.}, 2:665--741, 2002.
\newblock URL: \url{http://arxiv.org/abs/math/0103190}, \href
  {http://dx.doi.org/10.2140/agt.2002.2.665}
  {\path{doi:10.2140/agt.2002.2.665}}.

\bibitem{kl1}
M.~Khovanov$\phantom{a}\!\!\!$ and A.D. Lauda.
\newblock A diagrammatic approach to categorification of quantum groups {I}.
\newblock {\em Represent. Theory}, 13:309--347, 2009.
\newblock URL: \url{http://arxiv.org/abs/0803.4121}, \href
  {http://dx.doi.org/10.1090/S1088-4165-09-00346-X}
  {\path{doi:10.1090/S1088-4165-09-00346-X}}.

\bibitem{kl3}
M.~Khovanov$\phantom{a}\!\!\!$ and A.D. Lauda.
\newblock A diagrammatic approach to categorification of quantum groups {II}.
\newblock {\em Trans. Amer. Math. Soc.}, 363(5):2685--2700, 2011.
\newblock URL: \url{http://arxiv.org/abs/0804.2080}, \href
  {http://dx.doi.org/10.1090/S0002-9947-2010-05210-9}
  {\path{doi:10.1090/S0002-9947-2010-05210-9}}.

\bibitem{koi}
K.~Koike.
\newblock On the decomposition of tensor products of the representations of the
  classical groups: by means of the universal characters.
\newblock {\em Adv. Math.}, 74(1):57--86, 1989.
\newblock \href {http://dx.doi.org/10.1016/0001-8708(89)90004-2}
  {\path{doi:10.1016/0001-8708(89)90004-2}}.

\bibitem{kx}
S.~K{\"o}nig and C.~Xi.
\newblock On the structure of cellular algebras.
\newblock In {\em Algebras and modules, {II} ({G}eiranger, 1996)}, volume~24 of
  {\em CMS Conf. Proc.}, pages 365--386. Amer. Math. Soc., Providence, RI,
  1998.

\bibitem{kup}
G.~Kuperberg.
\newblock Spiders for rank {$2$} {L}ie algebras.
\newblock {\em Comm. Math. Phys.}, 180(1):109--151, 1996.
\newblock URL: \url{http://arxiv.org/abs/q-alg/9712003}.

\bibitem{llt}
A.~Lascoux, B.~Leclerc, and J.-Y. Thibon.
\newblock Hecke algebras at roots of unity and crystal bases of quantum affine
  algebras.
\newblock {\em Comm. Math. Phys.}, 181(1):205--263, 1996.

\bibitem{lz2}
G.~Lehrer and R.~Zhang.
\newblock The second fundamental theorem of invariant theory for the orthogonal
  group.
\newblock {\em Ann. of Math. (2)}, 176(3):2031--2054, 2012.
\newblock URL: \url{http://arxiv.org/abs/1102.3221}, \href
  {http://dx.doi.org/10.4007/annals.2012.176.3.12}
  {\path{doi:10.4007/annals.2012.176.3.12}}.

\bibitem{li}
G.~Li.
\newblock A {KLR} grading of the {B}rauer algebras.
\newblock 2014.
\newblock URL: \url{http://arxiv.org/abs/1409.1195}.

\bibitem{lu1}
G.~Lusztig$\phantom{a}\!\!\!$.
\newblock Modular representations and quantum groups.
\newblock In {\em Classical groups and related topics ({B}eijing, 1987)},
  volume~82 of {\em Contemp. Math.}, pages 59--77. Amer. Math. Soc.,
  Providence, RI, 1989.
\newblock \href {http://dx.doi.org/10.1090/conm/082/982278}
  {\path{doi:10.1090/conm/082/982278}}.

\bibitem{mack1}
M.~Mackaay.
\newblock The {$\mathfrak{sl}_n$}-web algebras and dual canonical bases.
\newblock {\em J. Algebra}, 409:54--100, 2014.
\newblock URL: \url{http://arxiv.org/abs/1308.0566}, \href
  {http://dx.doi.org/10.1016/j.jalgebra.2014.02.036}
  {\path{doi:10.1016/j.jalgebra.2014.02.036}}.

\bibitem{mpt}
M.~Mackaay, W.~Pan, and D.~Tubbenhauer.
\newblock The {$\mathfrak{sl}_3$}-web algebra.
\newblock {\em Math. Z.}, 277(1-2):401--479, 2014.
\newblock URL: \url{http://arxiv.org/abs/1206.2118}, \href
  {http://dx.doi.org/10.1007/s00209-013-1262-6}
  {\path{doi:10.1007/s00209-013-1262-6}}.

\bibitem{martin}
P.~Martin.
\newblock {\em Potts models and related problems in statistical mechanics},
  volume~5 of {\em Series on Advances in Statistical Mechanics}.
\newblock World Scientific Publishing Co., Inc., Teaneck, NJ, 1991.
\newblock \href {http://dx.doi.org/10.1142/0983} {\path{doi:10.1142/0983}}.

\bibitem{ms}
P.~Martin and H.~Saleur.
\newblock The blob algebra and the periodic {T}emperley-{L}ieb algebra.
\newblock {\em Lett. Math. Phys.}, 30(3):189--206, 1994.
\newblock URL: \url{http://arxiv.org/abs/hep-th/9302094}, \href
  {http://dx.doi.org/10.1007/BF00805852} {\path{doi:10.1007/BF00805852}}.

\bibitem{mas}
V.~Mazorchuk and C.~Stroppel.
\newblock {$G(\ell,k,d)$}-modules via groupoids.
\newblock {\em J. Algebraic Combin.}, 43(1):11--32, 2016.
\newblock URL: \url{http://arxiv.org/abs/1412.4494}, \href
  {http://dx.doi.org/10.1007/s10801-015-0623-0}
  {\path{doi:10.1007/s10801-015-0623-0}}.

\bibitem{mu}
J.~Murakami.
\newblock The {K}auffman polynomial of links and representation theory.
\newblock {\em Osaka J. Math.}, 24(4):745--758, 1987.

\bibitem{pag}
R.~Paget.
\newblock Representation theory of {$q$}-rook monoid algebras.
\newblock {\em J. Algebraic Combin.}, 24(3):239--252, 2006.
\newblock \href {http://dx.doi.org/10.1007/s10801-006-0010-y}
  {\path{doi:10.1007/s10801-006-0010-y}}.

\bibitem{par}
J.~Paradowski.
\newblock Filtrations of modules over the quantum algebra.
\newblock In {\em Algebraic groups and their generalizations: quantum and
  infinite-dimensional methods ({U}niversity {P}ark, {PA}, 1991)}, volume~56 of
  {\em Proc. Sympos. Pure Math.}, pages 93--108. Amer. Math. Soc., Providence,
  RI, 1994.

\bibitem{prh}
D.~Plaza and S.~Ryom-Hansen.
\newblock Graded cellular bases for {T}emperley--{L}ieb algebras of type {$A$}
  and {$B$}.
\newblock {\em J. Algebraic Combin.}, 40(1):137--177, 2014.
\newblock URL: \url{http://arxiv.org/abs/1203.2592}, \href
  {http://dx.doi.org/10.1007/s10801-013-0481-6}
  {\path{doi:10.1007/s10801-013-0481-6}}.

\bibitem{rw}
S.~Riche and G.~Williamson.
\newblock Tilting modules and the {$p$}-canonical basis.
\newblock 2015.
\newblock
  \url{https://hal-clermont-univ.archives-ouvertes.fr/hal-01249796/document}.
\newblock URL: \url{http://arxiv.org/abs/1512.08296}.

\bibitem{rsa}
D.~Ridout and Y.~Saint-Aubin.
\newblock Standard modules, induction and the structure of the
  {T}emperley--{L}ieb algebra.
\newblock {\em Adv. Theor. Math. Phys.}, 18(5):957--1041, 2014.
\newblock URL: \url{http://arxiv.org/abs/1204.4505}.

\bibitem{ring}
C.M. Ringel.
\newblock The category of modules with good filtrations over a quasi-hereditary
  algebra has almost split sequences.
\newblock {\em Math. Z.}, 208(2):209--223, 1991.
\newblock \href {http://dx.doi.org/10.1007/BF02571521}
  {\path{doi:10.1007/BF02571521}}.

\bibitem{rt}
D.E.V. Rose and D.~Tubbenhauer.
\newblock Symmetric webs, {J}ones--{W}enzl recursions, and {$q$}-{H}owe
  duality.
\newblock {\em Int. Math. Res. Not. IMRN}, (17):5249--5290, 2016.
\newblock URL: \url{http://arxiv.org/abs/1501.00915}, \href
  {http://dx.doi.org/10.1093/imrn/rnv302} {\path{doi:10.1093/imrn/rnv302}}.

\bibitem{rou}
R.~Rouquier.
\newblock {$2$}-{K}ac--{M}oody algebras.
\newblock 2008.
\newblock URL: \url{http://arxiv.org/abs/0812.5023}.

\bibitem{rsvv}
R.~Rouquier, P.~Shan, M.~Varagnolo, and E.~Vasserot.
\newblock Categorifications and cyclotomic rational double affine {H}ecke
  algebras.
\newblock {\em Invent. Math.}, 204(3):671--786, 2016.
\newblock URL: \url{http://arxiv.org/abs/1305.4456}, \href
  {http://dx.doi.org/10.1007/s00222-015-0623-7}
  {\path{doi:10.1007/s00222-015-0623-7}}.

\bibitem{rh3}
S.~Ryom-Hansen.
\newblock Cell structures on the blob algebra.
\newblock {\em Represent. Theory}, 16:540--567, 2012.
\newblock URL: \url{http://arxiv.org/abs/0911.1923}, \href
  {http://dx.doi.org/10.1090/S1088-4165-2012-00424-1}
  {\path{doi:10.1090/S1088-4165-2012-00424-1}}.

\bibitem{rh}
S.~Ryom-Hansen$\phantom{a}\!\!\!$.
\newblock The {A}riki--{T}erasoma--{Y}amada tensor space and the blob algebra.
\newblock {\em J. Algebra}, 324(10):2658--2675, 2010.
\newblock URL: \url{http://arxiv.org/abs/math/0505278}, \href
  {http://dx.doi.org/10.1016/j.jalgebra.2010.08.018}
  {\path{doi:10.1016/j.jalgebra.2010.08.018}}.

\bibitem{sash}
M.~Sakamoto and T.~Shoji.
\newblock Schur--{W}eyl reciprocity for {A}riki--{K}oike algebras.
\newblock {\em J. Algebra}, 221(1):293--314, 1999.
\newblock \href {http://dx.doi.org/10.1006/jabr.1999.7973}
  {\path{doi:10.1006/jabr.1999.7973}}.

\bibitem{sa}
A.~Sartori.
\newblock The degenerate affine walled {B}rauer algebra.
\newblock {\em J. Algebra}, 417:198--233, 2014.
\newblock URL: \url{http://arxiv.org/abs/1305.2347}, \href
  {http://dx.doi.org/10.1016/j.jalgebra.2014.06.030}
  {\path{doi:10.1016/j.jalgebra.2014.06.030}}.

\bibitem{sast}
A.~Sartori and C.~Stroppel.
\newblock Walled {B}rauer algebras as idempotent truncations of level {$2$}
  cyclotomic quotients.
\newblock {\em J. Algebra}, 440:602--638, 2015.
\newblock URL: \url{http://arxiv.org/abs/1411.2771}, \href
  {http://dx.doi.org/10.1016/j.jalgebra.2015.06.018}
  {\path{doi:10.1016/j.jalgebra.2015.06.018}}.

\bibitem{soe4}
W.~Soergel.
\newblock Character formulas for tilting modules over {K}ac--{M}oody algebras.
\newblock {\em Represent. Theory}, 2:432--448 (electronic), 1998.
\newblock \href {http://dx.doi.org/10.1090/S1088-4165-98-00057-0}
  {\path{doi:10.1090/S1088-4165-98-00057-0}}.

\bibitem{soe5}
W.~Soergel.
\newblock Character formulas for tilting modules over quantum groups at roots
  of one.
\newblock In {\em Current developments in mathematics, 1997 ({C}ambridge,
  {MA})}, pages 161--172. Int. Press, Boston, MA, 1999.

\bibitem{soe1}
W.~Soergel$\phantom{a}\!\!\!$.
\newblock Kategorie {$\mathcal{O}$}, perverse {G}arben und {M}oduln \"uber den
  {K}oinvarianten zur {W}eylgruppe.
\newblock {\em J. Amer. Math. Soc.}, 3(2):421--445, 1990.
\newblock \href {http://dx.doi.org/10.2307/1990960}
  {\path{doi:10.2307/1990960}}.

\bibitem{soe3}
W.~Soergel$\phantom{a}\!\!\!$.
\newblock Kazhdan--{L}usztig polynomials and a combinatoric[s] for tilting
  modules.
\newblock {\em Represent. Theory}, 1:83--114 (electronic), 1997.
\newblock \href {http://dx.doi.org/10.1090/S1088-4165-97-00021-6}
  {\path{doi:10.1090/S1088-4165-97-00021-6}}.

\bibitem{sol}
L.~Solomon.
\newblock The {B}ruhat decomposition, {T}its system and {I}wahori ring for the
  monoid of matrices over a finite field.
\newblock {\em Geom. Dedicata}, 36(1):15--49, 1990.
\newblock \href {http://dx.doi.org/10.1007/BF00181463}
  {\path{doi:10.1007/BF00181463}}.

\bibitem{st2}
C.~Stroppel.
\newblock Category {$\mathcal{O}$}: gradings and translation functors.
\newblock {\em J. Algebra}, 268(1):301--326, 2003.
\newblock \href {http://dx.doi.org/10.1016/S0021-8693(03)00308-9}
  {\path{doi:10.1016/S0021-8693(03)00308-9}}.

\bibitem{tani}
T.~Tanisaki.
\newblock Character formulas of {K}azhdan--{L}usztig type.
\newblock In {\em Representations of finite dimensional algebras and related
  topics in {L}ie theory and geometry}, volume~40 of {\em Fields Inst.
  Commun.}, pages 261--276. Amer. Math. Soc., Providence, RI, 2004.

\bibitem{tl}
H.N.V. Temperley and E.H. Lieb.
\newblock Relations between the ``percolation'' and ``colouring'' problem and
  other graph-theoretical problems associated with regular planar lattices:
  some exact results for the ``percolation'' problem.
\newblock {\em Proc. Roy. Soc. London Ser. A}, 322(1549):251--280, 1971.

\bibitem{tub3}
D.~Tubbenhauer.
\newblock {$\mathfrak{sl}_3$}-web bases, intermediate crystal bases and
  categorification.
\newblock {\em J. Algebraic Combin.}, 40(4):1001--1076, 2014.
\newblock URL: \url{https://arxiv.org/abs/1310.2779}, \href
  {http://dx.doi.org/10.1007/s10801-014-0518-5}
  {\path{doi:10.1007/s10801-014-0518-5}}.

\bibitem{tub4}
D.~Tubbenhauer$\phantom{a}\!\!\!$.
\newblock {$\mathfrak{sl}_n$}-webs, categorification and {K}hovanov--{R}ozansky
  homologies.
\newblock 2014.
\newblock URL: \url{http://arxiv.org/abs/1404.5752}.

\bibitem{tur1}
V.G. Turaev.
\newblock Operator invariants of tangles, and {$R$}-matrices.
\newblock {\em Izv. Akad. Nauk SSSR Ser. Mat.}, 53(5):1073--1107, 1135, 1989.
\newblock Translation in Math. USSR-Izv. 35:2 (1990), 411-444.

\bibitem{tur}
V.G. Turaev.
\newblock {\em Quantum invariants of knots and {$3$}-manifolds}, volume~18 of
  {\em de Gruyter Studies in Mathematics}.
\newblock Walter de Gruyter \& Co., Berlin, revised edition, 2010.
\newblock \href {http://dx.doi.org/10.1515/9783110221848}
  {\path{doi:10.1515/9783110221848}}.

\bibitem{wenzl}
H.~Wenzl.
\newblock On sequences of projections.
\newblock {\em C. R. Math. Rep. Acad. Sci. Canada}, 9(1):5--9, 1987.

\bibitem{west1}
B.W. Westbury.
\newblock Invariant tensors and cellular categories.
\newblock {\em J. Algebra}, 321(11):3563--3567, 2009.
\newblock URL: \url{http://arxiv.org/abs/0806.4045}, \href
  {http://dx.doi.org/10.1016/j.jalgebra.2008.07.004}
  {\path{doi:10.1016/j.jalgebra.2008.07.004}}.

\bibitem{west}
B.W. Westbury$\phantom{a}\!\!\!$.
\newblock The representation theory of the {T}emperley--{L}ieb algebras.
\newblock {\em Math. Z.}, 219(4):539--565, 1995.
\newblock \href {http://dx.doi.org/10.1007/BF02572380}
  {\path{doi:10.1007/BF02572380}}.

\end{thebibliography}
\bigskip

\noindent H.H.A.: \texttt{h.haahr.andersen@gmail.com}
\smallskip

\noindent C.S.: \texttt{stroppel@math.uni-bonn.de}
\smallskip

\noindent
D.T.: \texttt{dtubben@math.uni-bonn.de}

\includepdf[pages={-}]{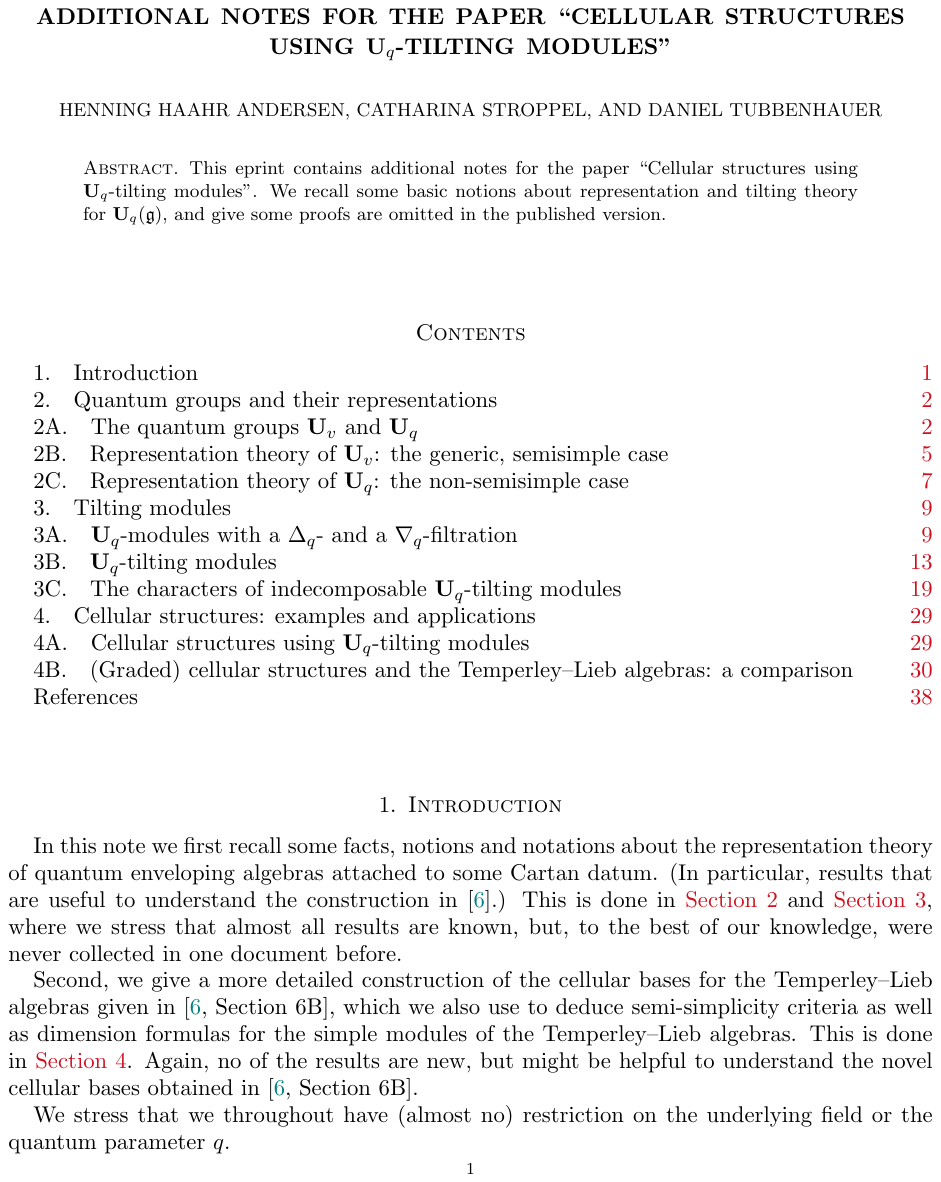}
\end{document}